\documentclass[11pt,reqno]{amsart}

\usepackage[T1]{fontenc}
\usepackage[a4paper,margin=1.08in]{geometry}
\usepackage{amsmath,amssymb,amsthm,mathtools}
\usepackage{enumitem}
\usepackage{microtype}
\usepackage[mathlines]{lineno}
\usepackage[colorlinks=true,citecolor=blue,linkcolor=blue,urlcolor=blue]{hyperref}

\hypersetup{
  pdftitle={Full-Tail Dynamical Rigidity Forced by Atomic Navier--Stokes Energy Concentration},
  pdfauthor={Hao Huang},
  pdfsubject={Navier--Stokes endpoint energy concentration and full-tail solenoidal saturation},
  pdfkeywords={Navier--Stokes equations, endpoint energy measure, concentration, solenoidal advection--diffusion, adjoint method}
}

\allowdisplaybreaks
\numberwithin{equation}{section}
\modulolinenumbers[5]

\newtheorem{theorem}{Theorem}[section]
\newtheorem{proposition}[theorem]{Proposition}
\newtheorem{lemma}[theorem]{Lemma}
\newtheorem{corollary}[theorem]{Corollary}
\newtheorem{definition}[theorem]{Definition}
\theoremstyle{remark}
\newtheorem{remark}[theorem]{Remark}

\newcommand{\T}{\mathbb T}
\newcommand{\R}{\mathbb R}
\newcommand{\Pp}{\mathbb P}
\newcommand{\Qq}{\mathbb Q}
\newcommand{\Hh}{\mathcal H}
\newcommand{\Kk}{\mathcal K}

\newcommand{\one}{\mathbf 1}
\newcommand{\dd}{\,\mathrm d}
\newcommand{\norm}[1]{\left\lVert#1\right\rVert}
\newcommand{\abs}[1]{\left|#1\right|}
\newcommand{\ip}[2]{\left\langle#1,#2\right\rangle}
\newcommand{\weakto}{\rightharpoonup}
\newcommand{\weakstarto}{\stackrel{*}{\rightharpoonup}}

\title[Atomic Endpoint Rigidity]{Full-Tail Dynamical Rigidity Forced by Atomic Navier--Stokes Energy Concentration}
\author{Hao Huang}
\address{Department of Computer Science and Engineering, Sungkyunkwan University, Suwon-si 16419, Gyeonggi-do, Republic of Korea}
\address{School of Mathematics and Statistics, Linyi University, Linyi 276000, China}
\email{huanghao92@skku.edu}


\begin{document}

\begin{abstract}
Let a smooth unforced three-dimensional Navier--Stokes flow on the flat torus
approach a finite terminal time.  Its kinetic-energy densities converge along
the full time variable to a unique endpoint measure.  We prove that each
point atom forces a same-parent, full-tail dynamical rigidity.  A preassigned
level-crossing catalogue and nested local-Hodge projections produce
orthogonal packets.  From the entire packet tail we extract one backward
adjoint whose terminal energy concentrates at the atom.  Cauchy saturation
then locks this adjoint to every late packet and yields uniform two-parameter
saturation of both the constrained Oseen propagator and its adjoint, with
vanishing first-order dissipation.  Every sufficiently late fixed-root
descendant consequently has infinite delayed second-order action and
nonintegrable positive enstrophy production.  Equivalently, a delayed
second-order operator budget determined solely by the Navier--Stokes parent
must fail arbitrarily close to the endpoint.
\end{abstract}

\maketitle

\section{Introduction}

Let \(\Omega=\T^3\) be the standard flat three-dimensional torus and let
\(\nu>0\).  We consider a real-valued divergence-free solution of
\begin{equation}
 \partial_tu+\Pp[(u\cdot\nabla)u]=\nu\Delta u,
 \qquad \nabla\cdot u=0,
 \label{eq:NS-intro}
\end{equation}
that is smooth on \([t_b,T_*)\times\Omega\), where \(T_*<\infty\).  The
Leray projector is denoted by \(\Pp\), and \(\Qq=I-\Pp\).  No maximality
assumption on \(T_*\) is used.  The principal application is the first
possible singular time of a smooth Navier--Stokes branch.

Endpoint energy measures encode possible loss of strong \(L^2\) compactness
without postulating a terminal state \(u(T_*)\).  Their concentration
properties were introduced by Shvydkoy and developed systematically by
Leslie and Shvydkoy~\cite{Shvydkoy2013,LeslieShvydkoy2018}.  In the present
smooth-preterminal viscous setting, the local energy identity determines a
unique finite measure \(\mu_*\) and, quantitatively,
\begin{equation}
 \abs{u(t,x)}^2\dd x\weakstarto\mu_*
 \qquad\text{as }t\uparrow T_*
 \label{eq:endpoint-measure-intro}
\end{equation}
along the full time variable; see Lemma~\ref{lem:unique-endpoint-measure}.
Related results constrain the size of concentration sets
\cite{ArnoldCraig2010}, establish partial regularity
\cite{CaffarelliKohnNirenberg1982}, give critical-space blow-up criteria
\cite{EscauriazaSereginSverak2003,Albritton2018}, or quantify concentration
of critical norms near singularities
\cite{BarkerPrange2020,BarkerPrange2021,MaekawaMiuraPrange2020}.  Under a
Type-I hypothesis, atomic concentration is excluded for three-dimensional
Euler flow~\cite{ChaeWolf2020}, while Type-I Navier--Stokes energy equality
is addressed in~\cite{LeslieShvydkoy2018}.

Taken together, these works describe the geometry, size, and
norm-theoretic cost of concentration and, under additional hypotheses, can
exclude atoms.  They do not identify the full-tail dynamical structure forced
by a single endpoint atom along the same nonlinear Navier--Stokes orbit.  The
present paper fills this structural gap.

This paper starts from the complementary situation
\begin{equation}
 \mu_*(\{a\})=m>0
 \label{eq:atom-intro}
\end{equation}
and asks what the \emph{same} Navier--Stokes orbit must do throughout its
entire late tail.  A point mass is a static endpoint datum; by itself it does
not distinguish isolated concentration events from a coherent preterminal
lineage.  Theorem~\ref{thm:atom-conveyor} proves that every atom necessarily
carries such a lineage.

Three features distinguish the result from existing concentration and linear
transport theories.  First, no Type-I, self-similar, smallness, or terminal
strong-convergence hypothesis is imposed.  Second, the atom does not merely
produce an adjoint or packet separately at each scale.  One backward adjoint,
extracted from the complete packet tail, locks to every sufficiently late
packet and yields estimates uniform over the full ordered triangle
\(k>j\geq J\).  Third, every propagator is driven by the original nonlinear
solution \(u\).  The drift is neither prescribed independently nor changed
from cell to cell.  This same-parent requirement separates the theorem from
mixing and small-scale-formation results for prescribed divergence-free
drifts~\cite{SilvestreVicol2012,Kumar2024,KumarWeber2026}, and from abstract
or exterior-domain evolution theories whose coefficients are fixed in
advance~\cite{HaakOuhabaz2015,AsamiHishida2025}.

The construction is ordered before compactness is invoked.  A countable
level-crossing catalogue is determined by the preterminal orbit without
using \(a\), \(m\), or \(\mu_*\).  Once an atom is specified, nested balls
are chosen from that catalogue.  Each new ball is frozen before the preceding
local-Hodge packet is defined, and a gap-two thinning preserves the already
constructed packets without reprojection.  Only then is the common adjoint
extracted.  Reverse-Oseen localization places its terminal energy at \(a\),
whereas the local-Hodge lower bound supplies at least the atomic mass.  The
full-time endpoint measure closes the two inequalities into equality in
Cauchy--Schwarz.  This saturation upgrades weak compactness to a unit adjoint
and then to uniform forward and backward transport saturation.

The second conclusion removes the atom-selected family from the statement.
For \(t_b\leq s<r<T_*\), let \(\mathfrak R_u(s,r)\) be the delayed
second-order evolution budget in \eqref{eq:parent-budget}.  It is defined
solely by the constrained Oseen propagator generated by \(u\).  Every
sufficiently late fixed-root descendant of the saturated family has
\begin{equation}
 \int_{\tau_{J+1}}^{T_*}
 \norm{\Delta U(t,\tau_J)q_J}_2^2\dd t=\infty
 \label{eq:intro-root-action}
\end{equation}
and nonintegrable positive enstrophy production.  Hence an atom forces
\(\mathfrak R_u(\tau_J,\tau_{J+1})=\infty\) for every sufficiently late
root.  Conversely, finiteness of this parent-only budget on one terminal tail
excludes point atoms.  The standard nonendpoint Serrin class is one natural
regime in which the budget is finite; see Corollary~\ref{cor:parent-budget}.

The principal chain is therefore
\begin{equation}
 \begin{aligned}
  &\text{one atom of the unique endpoint measure}
  \Longrightarrow \text{one same-parent full-tail saturated family}\\
  &\hspace{4.6cm}\Longrightarrow
  \text{failure of a parent-only delayed second-order budget}.
 \end{aligned}
 \label{eq:main-chain}
\end{equation}
The first implication is a structural rigidity theorem; the second converts
that structure into a family-free operator obstruction.  Together they
advance endpoint energy-measure theory from a static description of
concentration to same-parent dynamical rigidity and provide a concrete
operator target for endpoint compactness and atom-exclusion estimates.
Appendix
\ref{sec:endpoint-tests} records a further family-adapted endpoint test-space
consequence.

The paper is organized as follows.  Section~\ref{sec:setup} establishes the
full-time endpoint measure, the preterminal Oseen evolution family, and the
main statements.  Section~\ref{sec:level-crossing} constructs the frozen
local-Hodge chain.  Section~\ref{sec:linear} proves the two uniform linear
estimates.  Sections~\ref{sec:ghost} and \ref{sec:conveyor} extract the common
adjoint and prove full-tail saturation.  Section~\ref{sec:action} establishes
the second-order action and parent-only budget obstructions.  Section
\ref{sec:scope} summarizes the scope, and Appendix~\ref{sec:endpoint-tests}
contains the endpoint test-space consequence.

\section{Setting and main results}
\label{sec:setup}

We write \(L^2_\sigma(\Omega)\) for the real Hilbert space of periodic
divergence-free vector fields, including constant fields, with the usual
\(L^2\) inner product.  The solution \(u\) satisfies
\begin{equation}
 E_*:=\sup_{t_b\le t<T_*}\norm{u(t)}_2<\infty,
 \qquad
 \int_{t_b}^{T_*}\norm{\nabla u(t)}_2^2\dd t<\infty.
 \label{eq:energy-assumptions}
\end{equation}
These bounds follow from the energy equality for the smooth unforced branch.

\begin{lemma}[Quantitative full-time endpoint energy measure]
\label{lem:unique-endpoint-measure}
Under \eqref{eq:NS-intro} and \eqref{eq:energy-assumptions}, there is a
unique finite nonnegative Radon measure
\(\mu_*\in\mathcal M^+(\Omega)\) such that
\begin{equation}
 \abs{u(t,x)}^2\dd x\weakstarto\mu_*
 \qquad\text{as }t\uparrow T_*
 \label{eq:unique-endpoint-measure}
\end{equation}
along the full time variable.  More precisely, if
\begin{equation}
 D(s,t):=\int_s^t\norm{\nabla u(\rho)}_2^2\dd\rho,
 \label{eq:tail-enstrophy-D}
\end{equation}
then, for every \(\varphi\in C^\infty(\Omega)\) and
\(t_b\leq s<t<T_*\),
\begin{align}
 &\left|\int_\Omega\varphi\abs{u(t)}^2\dd x
       -\int_\Omega\varphi\abs{u(s)}^2\dd x\right|\notag\\
 &\quad\leq C_{\Omega,\varphi}\Bigl[
 E_*^{3/2}(t-s)^{1/4}
 \bigl(D(s,t)+E_*^2(t-s)\bigr)^{3/4}
 +\nu D(s,t)+\nu E_*^2(t-s)\Bigr].
 \label{eq:endpoint-measure-modulus}
\end{align}
 No maximality assumption on \(T_*\) is required.  The uniqueness conclusion
 is the periodic smooth-preterminal instance of the endpoint-measure theory in
 \cite{Shvydkoy2013,LeslieShvydkoy2018}; the displayed modulus is recorded for
 later full-time use.
\end{lemma}

\begin{proof}
Normalize the pressure by \(\int_\Omega p(t,x)\dd x=0\).  On the torus,
the pressure is represented by periodic Riesz transforms,
\begin{equation}
 p=\mathcal R_i\mathcal R_j(u_i u_j),
 \qquad
 \norm{p(t)}_{3/2}\leq C_\Omega\norm{u(t)}_3^2.
 \label{eq:endpoint-pressure-bound}
\end{equation}
For \(\varphi\in C^\infty(\Omega)\), set
\(F_\varphi(t)=\frac12\int_\Omega\varphi\abs{u(t)}^2\dd x\).
The local energy identity gives
\begin{align}
 F_\varphi'(t)
 ={}&\int_\Omega\left(\frac12\abs{u}^2+p\right)
 u\cdot\nabla\varphi\dd x
 +\frac\nu2\int_\Omega\abs{u}^2\Delta\varphi\dd x\notag\\
 &-\nu\int_\Omega\varphi\abs{\nabla u}^2\dd x.
 \label{eq:localized-energy-identity}
\end{align}
The additive spatial constant in the pressure is immaterial because
\(\int_\Omega u\cdot\nabla\varphi\dd x=0\).

Periodic Sobolev interpolation yields
\begin{equation}
 \norm{u(t)}_3^4
 \leq\norm{u(t)}_2^2\norm{u(t)}_6^2
 \leq C_\Omega E_*^2
 \bigl(\norm{\nabla u(t)}_2^2+E_*^2\bigr).
 \label{eq:L4tL3x-bound}
\end{equation}
Consequently, H\"older's inequality in time gives
\begin{align}
 \int_s^t\norm{u(\rho)}_3^3\dd\rho
 &\leq(t-s)^{1/4}
 \left(\int_s^t\norm{u(\rho)}_3^4\dd\rho\right)^{3/4}\notag\\
 &\leq C_\Omega E_*^{3/2}(t-s)^{1/4}
 \bigl(D(s,t)+E_*^2(t-s)\bigr)^{3/4}.
 \label{eq:cubic-tail-control}
\end{align}
The pressure flux has the same bound, since
\begin{equation}
 \int_\Omega\abs{p}\abs{u}\dd x
 \leq\norm{p}_{3/2}\norm{u}_3
 \leq C_\Omega\norm{u}_3^3.
 \label{eq:pressure-tail-control}
\end{equation}
Integrating \eqref{eq:localized-energy-identity}, using
\eqref{eq:cubic-tail-control}, and bounding the two viscous terms by
\(D(s,t)\) and \(E_*^2(t-s)\), proves
\eqref{eq:endpoint-measure-modulus}.

Because \(D(s,t)\to0\) whenever \(T_*>t>s\uparrow T_*\), the right-hand
side of \eqref{eq:endpoint-measure-modulus} tends to zero.  Thus
\(\int\varphi\abs{u(t)}^2\) has a unique terminal limit for every smooth
\(\varphi\).  Define
\begin{equation}
 L(\varphi):=\lim_{t\uparrow T_*}
 \int_\Omega\varphi\abs{u(t)}^2\dd x.
 \label{eq:endpoint-functional}
\end{equation}
This is a positive linear functional on \(C^\infty(\Omega)\) and
\(\abs{L(\varphi)}\leq E_*^2\norm{\varphi}_\infty\).  It extends uniquely
to a positive bounded functional on \(C(\Omega)\); the Riesz representation
theorem supplies a unique \(\mu_*\in\mathcal M^+(\Omega)\).  Finally, the
uniform mass bound and density of \(C^\infty(\Omega)\) in \(C(\Omega)\)
extend convergence to every continuous test function, proving
\eqref{eq:unique-endpoint-measure}.
\end{proof}

For \(t\ge s\), let \(U(t,s)\) denote the solenoidal
advection--diffusion propagator generated by the Navier--Stokes parent,
\begin{equation}
 \partial_th+\Pp[(u\cdot\nabla)h]=\nu\Delta h,
 \qquad h(s)=h_s\in L^2_\sigma(\Omega).
 \label{eq:constrained}
\end{equation}
Let \(S(t,s)\) be the componentwise passive propagator
\begin{equation}
 \partial_tz+(u\cdot\nabla)z=\nu\Delta z,
 \qquad z(s)=z_s\in L^2(\Omega;\R^3).
 \label{eq:passive}
\end{equation}
The passive evolution need not preserve divergence.  The backward adjoint of
\eqref{eq:constrained} is
\begin{equation}
 \partial_tA+\Pp[(u\cdot\nabla)A]=-\nu\Delta A.
 \label{eq:adjoint}
\end{equation}

\begin{proposition}[Preterminal Oseen evolution family]
\label{prop:oseen-family}
For \(t_b\leq s\leq r\leq t<T_*\), the operators \(U(t,s)\) form a
strongly continuous evolution family on \(L^2_\sigma(\Omega)\):
\begin{equation}
 U(t,r)U(r,s)=U(t,s),\qquad U(s,s)=I.
 \label{eq:oseen-cocycle}
\end{equation}
For every \(h_s\in L^2_\sigma(\Omega)\),
\begin{equation}
 \norm{U(t,s)h_s}_2^2
 +2\nu\int_s^t\norm{\nabla U(\rho,s)h_s}_2^2\dd\rho
 =\norm{h_s}_2^2.
 \label{eq:oseen-forward-energy}
\end{equation}
For terminal data \(g_t\in L^2_\sigma(\Omega)\), the function
\(A(\rho)=U(t,\rho)^*g_t\) is the unique energy solution of
\eqref{eq:adjoint} on \([s,t]\), and
\begin{equation}
 \norm{A(s)}_2^2+2\nu\int_s^t\norm{\nabla A(\rho)}_2^2\dd\rho
 =\norm{g_t}_2^2.
 \label{eq:oseen-adjoint-energy}
\end{equation}
The duality relation
\begin{equation}
 \ip{U(t,s)h_s}{g_t}=\ip{h_s}{U(t,s)^*g_t}
 \label{eq:oseen-duality}
\end{equation}
holds for all such data.  In particular, \(U(t,s)\), \(U(t,s)^*\), and
\(S(t,s)\) are \(L^2\)-contractions.

For every positive delay \(s<r<T_*\), one also has
\begin{equation}
 \norm{\nabla U(r,s)h_s}_2^2
 \leq \frac{C_\Omega}{\nu(r-s)}
 \exp\!\left(\frac{C_\Omega}{\nu}
       \int_s^r\norm{u(\rho)}_\infty^2\dd\rho\right)
 \norm{h_s}_2^2.
 \label{eq:oseen-delayed-H1}
\end{equation}
Thus \(U(r,s):L^2_\sigma(\Omega)\to H^1_\sigma(\Omega)\) is bounded.
\end{proposition}

\begin{proof}
Fix a compact preterminal interval.  Fourier--Galerkin approximation in the
solenoidal subspace gives finite-dimensional solutions of
\eqref{eq:constrained}.  The transport term is skew in \(L^2\), because
\(\nabla\cdot u=0\), and \(\Pp\) is the orthogonal projection onto
\(L^2_\sigma\).  Testing by the Galerkin solution therefore gives
\eqref{eq:oseen-forward-energy}.  Weak compactness, uniqueness in the energy
class, and the energy identity pass to the limit.  Uniqueness also gives the
cocycle law; the standard energy argument at the initial time gives strong
continuity.  Applying the same construction backward from terminal data,
or equivalently taking adjoints of the Galerkin matrices, gives
\eqref{eq:adjoint} and \eqref{eq:oseen-adjoint-energy}.  Differentiating
\(\ip{U(\rho,s)h_s}{A(\rho)}\) at the Galerkin level and passing to the
limit proves \eqref{eq:oseen-duality}.  The passive contraction follows by
testing \eqref{eq:passive} componentwise.

It remains to record the positive-delay estimate used later.  For smooth
data, testing \eqref{eq:constrained} by \(-\Delta h\) gives
\begin{equation}
 \frac12\frac{\dd}{\dd t}\norm{\nabla h}_2^2
 +\nu\norm{\Delta h}_2^2
 \leq \norm{u}_\infty\norm{\nabla h}_2\norm{\Delta h}_2
 \leq \frac\nu2\norm{\Delta h}_2^2
      +\frac{1}{2\nu}\norm{u}_\infty^2\norm{\nabla h}_2^2.
 \label{eq:oseen-H1-differential}
\end{equation}
By \eqref{eq:oseen-forward-energy}, there is
\(\rho_0\in(s,(s+r)/2)\) such that
\[
 \norm{\nabla h(\rho_0)}_2^2
 \leq [\nu(r-s)]^{-1}\norm{h_s}_2^2.
\]
Gronwall's inequality applied to
\eqref{eq:oseen-H1-differential} on \([\rho_0,r]\) proves
\eqref{eq:oseen-delayed-H1}.  Approximation of arbitrary \(L^2\) data by
smooth solenoidal data completes the proof.
\end{proof}

For an open ball \(B\Subset\Omega\), define the local solenoidal space
\begin{equation}
 \Hh(B):=\overline{C_{c,\sigma}^\infty(B)}^{L^2(\Omega)},
 \qquad Q_B:=\operatorname{Proj}_{\Hh(B)}.
 \label{eq:local-hodge-space}
\end{equation}
Every element of \(\Hh(B)\) vanishes almost everywhere on
\(\Omega\setminus B\) and has zero spatial mean.  Indeed, for
\(\phi\in C_{c,\sigma}^\infty(B)\), integration in a Euclidean chart gives
\(\int_\Omega\phi_i\dd x=\int_\Omega\operatorname{div}(x_i\phi)\dd x=0\);
the assertion then passes to the \(L^2\)-closure.

\begin{theorem}[Atomic concentration forces full-tail solenoidal saturation]
\label{thm:atom-conveyor}
Assume \eqref{eq:NS-intro} and \eqref{eq:energy-assumptions}, and suppose
that the unique measure in Lemma~\ref{lem:unique-endpoint-measure} satisfies
\eqref{eq:atom-intro}.  Then there exist times \(\tau_j\uparrow T_*\),
nested balls \(B_{j+1}\Subset B_j\downarrow\{a\}\), orthonormal vectors
\(q_j\in L^2_\sigma(\Omega)\), and one full-tail adjoint \(A\), with the
following properties.

\begin{enumerate}[label=\textup{(\alph*)},leftmargin=2.2em]
\item With \(\Hh_j=\Hh(B_j)\),
\begin{equation}
 q_j\in\Hh_j\cap\Hh_{j+1}^\perp,
 \qquad q_j=0\ \text{a.e. on }\Omega\setminus B_j,
 \label{eq:packet-geometry}
\end{equation}
and \(q_j\) is a harmonic gradient in \(B_{j+1}\).  Moreover,
\begin{equation}
 \ip{u(\tau_j)}{q_j}\longrightarrow\sqrt m.
 \label{eq:full-mass-capture}
\end{equation}

\item For every \(S<T_*\),
\[
 A\in C([t_b,S];L^2_\sigma(\Omega))
 \cap L^2(t_b,S;H^1_\sigma(\Omega)),
 \qquad
 \partial_tA\in L^2(t_b,S;H^{-1}_\sigma(\Omega)).
\]
It solves \eqref{eq:adjoint} distributionally: for every
\(\varphi\in C_c^\infty((t_b,S)\times\Omega;\mathbb R^3)\) with
\(\nabla\cdot\varphi=0\),
\[
 \int_{t_b}^{S}\!\!\int_\Omega
 \bigl(A\cdot\partial_t\varphi
       +A\cdot(u\cdot\nabla)\varphi
       +\nu\nabla A:\nabla\varphi\bigr)\dd x\dd t=0.
\]
It satisfies
\begin{align}
 \ip{u(t)}{A(t)}&=\sqrt m,
 &&t_b<t<T_*,
 \label{eq:constant-pairing}\\
 A(t)&\weakto0\quad\text{in }L^2,
 &&t\uparrow T_*,
 \label{eq:weak-zero}\\
 \abs{A(t,x)}^2\dd x&\weakstarto\delta_a,
 &&t\uparrow T_*.
 \label{eq:adjoint-atom}
\end{align}
The parent and adjoint satisfy the full-time atomic-layer alignment identity
\begin{equation}
 \lim_{r\downarrow0}\limsup_{t\uparrow T_*}
 \int_{B_r(a)}\abs{u(t)-\sqrt m\,A(t)}^2\dd x=0.
 \label{eq:atomic-alignment}
\end{equation}

\item The moving nodes lock to the common adjoint:
\begin{equation}
 \norm{A(\tau_j)-q_j}_2\longrightarrow0.
 \label{eq:moving-locking}
\end{equation}

\item The late triangular saturation is uniform in both directions:
\begin{align}
 \sup_{k>j\ge J}
 \norm{U(\tau_k,\tau_j)^*q_k-q_j}_2&\longrightarrow0,
 \label{eq:adjoint-conveyor}\\
 \sup_{k>j\ge J}
 \norm{U(\tau_k,\tau_j)q_j-q_k}_2&\longrightarrow0
 \label{eq:forward-conveyor}
\end{align}
as \(J\to\infty\).  In addition,
\begin{equation}
 \sup_{k>j\ge J}2\nu\int_{\tau_j}^{\tau_k}
 \norm{\nabla U(t,\tau_j)q_j}_2^2\dd t\longrightarrow0,
 \label{eq:vanishing-action}
\end{equation}
and
\begin{equation}
 \sup_{k>j\ge J}2\nu\int_{\tau_j}^{\tau_k}
 \norm{\nabla U(\tau_k,t)^*q_k}_2^2\dd t\longrightarrow0.
 \label{eq:vanishing-adjoint-action}
\end{equation}

\end{enumerate}
\end{theorem}

\begin{definition}[Full-tail saturated Hodge family]
\label{def:perfect-conveyor}
For a fixed Navier--Stokes parent \(u\), a sequence
\((\tau_j,B_j,q_j)\) is a saturated Hodge family if
\(\tau_j\) is strictly increasing with \(\tau_j\uparrow T_*\), the balls
satisfy \(B_{j+1}\Subset B_j\) and \(\operatorname{diam}B_j\to0\), and the
packets are orthonormal, satisfy \eqref{eq:packet-geometry}, are harmonic
gradients in \(B_{j+1}\), and obey \eqref{eq:forward-conveyor}.  The family is called full-tail saturated if it
also obeys
\eqref{eq:adjoint-conveyor}, \eqref{eq:vanishing-action},
and \eqref{eq:vanishing-adjoint-action}.
\end{definition}

\begin{corollary}[Constrained--passive separation]
\label{cor:passive-separation}
Under the hypotheses of Theorem~\ref{thm:atom-conveyor}, the saturated family
may be chosen so that, for some \(\eta_j\downarrow0\),
\begin{equation}
 \abs{\ip{S(\tau_j,\tau_i)q_i}{q_j}}\le\eta_j,
 \qquad i<j.
 \label{eq:passive-avoidance}
\end{equation}
Consequently, with \(U_j=U(\tau_{j+1},\tau_j)\) and
\(S_j=S(\tau_{j+1},\tau_j)\),
\begin{equation}
 \operatorname{Re}\ip{(U_j-S_j)q_j}{q_{j+1}}\longrightarrow1.
 \label{eq:leray-passive-separation}
\end{equation}
This is an operator-theoretic unit separation; it is not identified with a
pressure norm or a pressure measure.
\end{corollary}

For a fixed root \(J\), put
\begin{equation}
 H^J(t):=U(t,\tau_J)q_J,
 \qquad
 \Gamma_J(t):=\norm{\nabla H^J(t)}_2^2,
 \qquad
 \Kk_J(t):=\norm{\Delta H^J(t)}_2^2.
 \label{eq:root-definitions}
\end{equation}
The enstrophy production is
\begin{equation}
 \mathcal P_J(t)
 :=\ip{\Delta H^J(t)}{\Pp[(u(t)\cdot\nabla)H^J(t)]}.
 \label{eq:production-definition}
\end{equation}
Equivalently, integration by parts gives
\begin{equation}
 \mathcal P_J(t)
 =-\int_\Omega \partial_\ell u_k\,\partial_kH_i^J\,
 \partial_\ell H_i^J\dd x,
 \label{eq:production-tensor}
\end{equation}
and
\begin{equation}
 \frac12\Gamma_J'(t)+\nu\Kk_J(t)=\mathcal P_J(t).
 \label{eq:enstrophy-identity}
\end{equation}

For \(t_b\leq s<r<T_*\), define the delayed solenoidal
second-order action budget of the parent flow by
\begin{equation}
 \mathfrak R_u(s,r):=
 \sup_{\substack{h_s\in L^2_\sigma(\Omega)\\\norm{h_s}_2=1}}
 \int_r^{T_*}\norm{\Delta U(t,s)h_s}_2^2\dd t
 \in[0,\infty].
 \label{eq:parent-budget}
\end{equation}
This quantity is fixed by \(u,s,r\); it does not use an endpoint atom, a
packet family, or an adapted weight.  Although the witness pairs below arise
from the atomic catalogue, the budget evaluated at each pair is defined solely
by the parent propagator.  Thus \(\mathfrak R_u(s,r)^{1/2}\) is the
corresponding extended operator norm of
\(h_s\mapsto\Delta U(\cdot,s)h_s\), with the value \(+\infty\) allowed.
Whenever it is finite, this map is a bounded operator from
\(L^2_\sigma(\Omega)\) to \(L^2((r,T_*);L^2(\Omega))\), and its usual operator
norm is exactly \(\mathfrak R_u(s,r)^{1/2}\).  On the torus,
\(\norm{\Delta h}_2\) is the homogeneous second-order Sobolev norm of the
mean-free component; the spatial mean is preserved by \(U(t,s)\).

\begin{corollary}[Parent-only delayed second-order-action obstruction]
\label{cor:parent-budget}
Under the hypotheses of Theorem~\ref{thm:atom-conveyor}, every sufficiently
late fixed root satisfies
\begin{equation}
 \mathfrak R_u(\tau_J,\tau_{J+1})=\infty.
 \label{eq:parent-budget-failure}
\end{equation}
Indeed, the unit vector \(q_J\) itself realizes an infinite integral in
\eqref{eq:parent-budget}.  Consequently, if there exists \(t_c<T_*\) such
that
\begin{equation}
 \mathfrak R_u(s,r)<\infty
 \qquad\text{for every }t_c\leq s<r<T_*,
 \label{eq:parent-budget-criterion}
\end{equation}
then the endpoint measure \(\mu_*\) is atomless.  In particular, \eqref{eq:parent-budget-criterion} holds whenever
\begin{equation}
 u\in L^p(t_c,T_*;L^q(\Omega)),
 \qquad 3<q\leq\infty,
 \qquad \frac2p+\frac3q\leq1.
 \label{eq:serrin-budget-class}
\end{equation}
\end{corollary}

A secondary family-adapted endpoint test-space obstruction is stated and
proved in Section~\ref{sec:endpoint-tests}.

\section{A catalogued level-crossing chain and local Hodge packets}
\label{sec:level-crossing}

We first separate the selection mechanism from the later adjoint extraction.
Fix once and for all a countable dense set of centers in \(\Omega\), a
countable dense set of radii below the injectivity radius, rational buffer
ratios in \((0,1/8)\), positive rational levels \(\theta\), guard and high times from a fixed
countable dense subset of \((t_b,T_*)\), and a
cutoff template \(\chi_0\in C_c^\infty(B_2(0))\) with
\(0\le\chi_0\le1\) and \(\chi_0=1\) on \(B_1(0)\).  For admissible
\((z,R,\beta)\), write
\begin{equation}
 \chi_{z,R,\beta}(x)
 =\chi_0\!\left(\frac{\exp_z^{-1}x}{\beta R}\right),
 \qquad
 M_\chi(t)=\int_\Omega\chi\abs{u(t)}^2\dd x.
 \label{eq:catalog-cutoff}
\end{equation}
Only parameters for which the support lies in a member of a fixed finite
Euclidean atlas are retained.

Given guard and high times \(h<H<T_*\), an item is eligible when
\begin{equation}
 M_\chi(h)<\theta<M_\chi(H).
 \label{eq:catalog-crossing}
\end{equation}
For an eligible item define the first level crossing
\begin{equation}
 \tau=\min\{t\in[h,H]:M_\chi(t)=\theta\}.
 \label{eq:first-crossing}
\end{equation}
For every time \(\sigma\in(\tau,H)\) belonging to the same dense
preterminal set, the augmented tuple is retained as a catalogue item.  The
resulting catalogue
is countable and is determined by the preterminal orbit and the fixed
parameter skeleton.  It is defined without using \(\mu_*,a\), or \(m\).

\begin{lemma}[Uniform absolute continuity on compact orbit segments]
\label{lem:uniform-ac}
For every \(h<T_*\), the set
\(\{u(t):t_b\le t\le h\}\) is compact in \(L^2(\Omega)\).  Consequently,
if measurable sets \(E_n\subset\Omega\) satisfy \(\abs{E_n}\to0\), then
\begin{equation}
 \sup_{t_b\le t\le h}\norm{u(t)}_{L^2(E_n)}\longrightarrow0.
 \label{eq:uniform-ac}
\end{equation}
\end{lemma}

\begin{proof}
The first assertion follows from the strong \(L^2\)-continuity of the smooth
preterminal trajectory.  Cover its compact image by a finite \(L^2\)-net.
The integral of each net point is absolutely continuous in the spatial
measure; taking the maximum over the finite net and absorbing twice the net
error gives \eqref{eq:uniform-ac}.
\end{proof}

\begin{lemma}[Nested local-Hodge structure]
\label{lem:nested-hodge}
Let \(B'\Subset B\Subset\Omega\) be balls contained in a Euclidean chart.
Then \(\Hh(B')\subset\Hh(B)\), and their orthogonal projections satisfy
\begin{equation}
 Q_{B'}Q_B=Q_BQ_{B'}=Q_{B'}.
 \label{eq:nested-projections}
\end{equation}
Consequently,
\begin{equation}
 D_{B,B'}:=Q_B-Q_{B'}
 \label{eq:hodge-difference-projection}
\end{equation}
is the orthogonal projection onto
\(\Hh(B)\cap\Hh(B')^\perp\), and
\begin{equation}
 \norm{D_{B,B'}v}_2^2
 =\norm{Q_Bv}_2^2-\norm{Q_{B'}v}_2^2
 \qquad(v\in L^2_\sigma(\Omega)).
 \label{eq:hodge-pythagoras}
\end{equation}
Moreover, if \(q\in\Hh(B)\cap\Hh(B')^\perp\), then on \(B'\) there is a
harmonic distribution \(h\), smooth in the interior, such that
\(q=\nabla h\).
\end{lemma}

\begin{proof}
The inclusion follows directly from the definitions.  If \(v\in L^2\), then
\(v-Q_Bv\perp\Hh(B)\), hence also \(v-Q_Bv\perp\Hh(B')\); this gives
\(Q_{B'}Q_Bv=Q_{B'}v\).  The identity \(Q_BQ_{B'}=Q_{B'}\) follows from
\(\Hh(B')\subset\Hh(B)\).  Thus \(D_{B,B'}\) is self-adjoint and
idempotent, with the stated range, and \eqref{eq:hodge-pythagoras} follows
from the orthogonal decomposition
\(Q_Bv=D_{B,B'}v+Q_{B'}v\).

If \(q\perp\Hh(B')\), then its restriction to \(B'\) annihilates every
compactly supported smooth solenoidal test field.  The local de Rham theorem
gives \(q=\nabla h\) in \(B'\).  Since \(q\) is divergence-free,
\(\Delta h=0\) distributionally; interior elliptic regularity makes \(h\)
smooth.
\end{proof}

\begin{lemma}[Interior local-Hodge estimate]
\label{lem:hodge-interior}
Let \(B_r\Subset B_R\Subset\Omega\) be concentric balls contained in a
Euclidean chart.  For every \(v\in L^2_\sigma(\Omega)\),
\begin{equation}
 \norm{Q_{B_R}v}_2
 \ge \norm{v}_{L^2(B_r)}
 -C\left(\frac rR\right)^{3/2}\norm{v}_{L^2(B_R)}.
 \label{eq:hodge-interior}
\end{equation}
The constant is uniform for balls below a fixed fraction of the injectivity
radius.
\end{lemma}

\begin{proof}
Set \(w=Q_{B_R}v\).  The restriction of \(v-w\) to \(B_R\) annihilates
all compactly supported smooth solenoidal test fields.  Hence
\(v-w=\nabla h\) in \(B_R\) for a harmonic distribution \(h\).  Because
\(w\in\Hh(B_R)\), \(w=0\) almost everywhere outside \(B_R\), and the
projection orthogonality gives
\begin{equation}
 \norm{\nabla h}_{L^2(B_R)}^2
 =\norm{v}_{L^2(B_R)}^2-\norm{w}_2^2
 \leq\norm{v}_{L^2(B_R)}^2.
 \label{eq:hodge-error}
\end{equation}
The scale-invariant interior estimate for a harmonic gradient yields
\[
 \norm{\nabla h}_{L^2(B_r)}
 \leq C(r/R)^{3/2}\norm{\nabla h}_{L^2(B_R)}.
\]
Since \(v=w+\nabla h\) in \(B_R\), the triangle inequality and
\eqref{eq:hodge-error} prove \eqref{eq:hodge-interior}.
\end{proof}

\begin{lemma}[Persistence under frozen thinning]
\label{lem:frozen-thinning}
Suppose \((B_n,q_n)\) is a nested chain with
\[
 q_n\in\Hh(B_n)\cap\Hh(B_{n+1})^\perp,
 \qquad q_n=\nabla h_n\ \text{in }B_{n+1},
\]
and suppose the mass, old-time smallness, and passive-avoidance estimates at
the original indices have already been fixed.  Let \((n_j)\) be strictly
increasing with \(n_{j+1}\geq n_j+2\), and retain the existing objects
\[
 \widetilde B_j=B_{n_j},\qquad \widetilde q_j=q_{n_j},
 \qquad \widetilde\tau_j=\tau_{n_j}.
\]
Then
\begin{equation}
 \widetilde q_j\in\Hh(\widetilde B_j)
 \cap\Hh(\widetilde B_{j+1})^\perp,
 \qquad
 \widetilde q_j\ \text{is a harmonic gradient in }\widetilde B_{j+1}.
 \label{eq:thinned-hodge-inheritance}
\end{equation}
The retained packets remain orthonormal, and all estimates attached to an
individual original index remain unchanged.  If the original passive
estimate holds whenever \(i\leq j-2\), it holds for every ordered pair of
retained indices.  The same conclusions persist under every further
subsequence.  No retained packet is redefined or reprojected.
\end{lemma}

\begin{proof}
Nestedness and \(n_{j+1}\geq n_j+2\) give
\(\Hh(B_{n_{j+1}})\subset\Hh(B_{n_j+1})\).  Hence the original
orthogonality to \(\Hh(B_{n_j+1})\) implies orthogonality to the next
retained Hodge space.  Also
\(B_{n_{j+1}}\Subset B_{n_j+1}\), so the original harmonic-gradient
identity restricts to the next retained ball.  Orthonormality and all
single-index estimates are unchanged.  For retained indices \(i<j\), one
has \(n_i\leq n_j-2\), which is precisely the original passive-avoidance
range.  Repeating the same inclusion argument proves persistence under any
further subsequence.
\end{proof}

\begin{proposition}[Catalogued level-crossing selection]
\label{prop:catalog-selection}
Assume \eqref{eq:atom-intro}.  Let
\(\gamma_j\uparrow T_*\), \(\delta_j^{\rm given}\downarrow0\), and
\(\eta_j\downarrow0\) be prescribed.  Then one can select from the fixed
catalogue a sequence of entries, rational levels \(\theta_j\uparrow m\), and
tolerances chosen recursively after the levels so that
\begin{equation}
 0<\widehat\delta_j\le
 \min\{\delta_j^{\rm given},2^{-j},\tfrac14\sqrt{\theta_j},
       \tfrac12\widehat\delta_{j-1}\},
 \label{eq:delta-hat}
\end{equation}
where the last entry is omitted for \(j=1\), such that the following hold.

There are balls \(B_j=B_{R_j}(z_j)\), inner radii
\(r_j=\beta_jR_j\), and cutoffs
\(\chi_j\) such that
\begin{align}
 &a\in B_{r_j}(z_j),
 \quad \operatorname{supp}\chi_j\subset B_{2r_j}(z_j)\Subset B_j,
 \quad B_{j+1}\Subset B_j,
 \label{eq:nested-balls}\\
 &R_j<2^{-j},
 \quad R_j\downarrow0,
 \quad \bigcap_j\overline{B_j}=\{a\},
 \label{eq:balls-to-a}\\
 &\tau_j<\sigma_j<\tau_{j+1},
 \quad \tau_j>\gamma_j,
 \label{eq:ordered-times}\\
 &\int_\Omega\chi_j\abs{u(\tau_j)}^2\dd x=\theta_j,
 \label{eq:selected-crossing}\\
 &\sup_{t_b\le t\le\sigma_j}
 \norm{u(t)}_{L^2(B_{j+1})}\le\widehat\delta_j.
 \label{eq:selected-old-small}
\end{align}

Let \(Q_j=Q_{B_j}\), \(D_j=Q_j-Q_{j+1}\), and retain the old packet
\begin{equation}
 q_j=\frac{D_ju(\tau_j)}{\norm{D_ju(\tau_j)}_2}.
 \label{eq:q-definition}
\end{equation}
Then \(q_j\) is well-defined, the packets are orthonormal,
\begin{equation}
 q_j\in\Hh(B_j)\cap\Hh(B_{j+1})^\perp,
 \qquad
 \ip{u(\tau_j)}{q_j}^2
 \ge\theta_j-2\widehat\delta_j\sqrt{\theta_j},
 \label{eq:hodge-full-mass}
\end{equation}
and \(q_j\) is a harmonic gradient in \(B_{j+1}\).  The same construction
can be made to satisfy
\begin{equation}
 \abs{\ip{S(\tau_j,\tau_i)q_i}{q_j}}\le\eta_j,
 \qquad i\le j-2.
 \label{eq:pre-avoidance}
\end{equation}
After the explicit gap-two thinning in the proof, all conclusions persist under
arbitrary further subsequences in the sense of
Lemma~\ref{lem:frozen-thinning}.  In particular, no packet is redefined or
reprojected after thinning.
\end{proposition}

\begin{proof}
We give the selection order because it carries the main quantifier content.
Choose rational \(\theta_j\uparrow m\) with \(\theta_j<m\), choose
\(\widehat\delta_j\) recursively according to \eqref{eq:delta-hat}, and then
choose rational buffers
\(\beta_j\downarrow0\) so small that
\[
 C(2\beta_j)^{3/2}E_*\leq\widehat\delta_j.
\]
Thus the error term in Lemma~\ref{lem:hodge-interior}, with
\(r=2\beta_jR\), is at most \(\widehat\delta_j\) for every \(R\).

Choose a chart ball \(B_0\) containing \(a\), and set
\(\sigma_0=t_b\) and \(\widehat\delta_0=1\).  Suppose levels through
\(j-1\) and packets through \(j-2\) have been fixed.
The packet \(q_{j-1}\) remains pending until \(B_j\) has been fixed.  Choose a
rational guard time
\begin{equation}
 h_j>\max\{\sigma_{j-1},\gamma_j\}.
 \label{eq:guard-time}
\end{equation}
By Lemma~\ref{lem:uniform-ac}, density of the center-radius skeleton, and the
fact that \(a\) lies in the interior of the previous ball, choose
\(B_j=B_{R_j}(z_j)\Subset B_{j-1}\) with \(R_j<2^{-j}\),
\(a\in B_{\beta_jR_j}(z_j)\), and
\begin{equation}
 \sup_{t_b\le t\le h_j}\norm{u(t)}_{L^2(B_j)}
 \le\min\{\widehat\delta_{j-1},\tfrac12\sqrt{\theta_j}\}.
 \label{eq:guard-smallness}
\end{equation}
At this stage only finitely many old passive trajectories are present.
The drift-independent estimate proved in
Lemma~\ref{lem:nash-smoothing} below makes their \(L^\infty\) norms uniformly
bounded for \(t\ge h_j\).  Shrinking \(B_j\) further, before defining the new
packet, therefore ensures
\begin{equation}
 \abs{B_j}^{1/2}
 \max_{i\le j-2}\sup_{t\ge h_j}
 \norm{S(t,\tau_i)q_i}_\infty\le\eta_j.
 \label{eq:avoidance-selection}
\end{equation}
For \(j\geq2\), once this final choice of \(B_j\) has been made, define the
previously pending packet \(q_{j-1}\) by \eqref{eq:q-definition}, with
\(D_{j-1}=Q_{j-1}-Q_j\).  The packet \(q_j\) remains pending until
\(B_{j+1}\) is fixed.  Thus each packet is defined exactly once, after both
Hodge spaces entering its difference projection have been fixed.

Since \(\chi_j(a)=1\), Lemma~\ref{lem:unique-endpoint-measure} gives the
full-time limit
\begin{equation}
 M_{\chi_j}(t)
 \longrightarrow\int\chi_j\dd\mu_*\ge m
 \qquad\text{as }t\uparrow T_*.
 \label{eq:atom-witness}
\end{equation}
Because \(\theta_j<m\), choose a high time \(H_j>h_j\) from the fixed dense
set such that \(M_{\chi_j}(H_j)>\theta_j\).  The guard smallness in
\eqref{eq:guard-smallness} implies
\(M_{\chi_j}(h_j)<\theta_j\), so this is an eligible catalogue item.
Definition~\eqref{eq:first-crossing} gives \(\tau_j\in(h_j,H_j)\), and we
choose \(\sigma_j\in(\tau_j,H_j)\) from the fixed dense set.  The next guard
time lies after \(\sigma_j\), which proves \eqref{eq:ordered-times}.  Equation
\eqref{eq:guard-smallness} at the next stage gives
\eqref{eq:selected-old-small}.

Apply Lemma~\ref{lem:hodge-interior} to the inner ball containing
\(\operatorname{supp}\chi_j\) and the outer ball \(B_j\).  The buffer choice
and \eqref{eq:selected-crossing} yield
\begin{equation}
 \norm{Q_ju(\tau_j)}_2\ge\sqrt{\theta_j}-\widehat\delta_j.
 \label{eq:Qj-lower}
\end{equation}
On the other hand, \eqref{eq:selected-old-small} and the support of
\(\Hh(B_{j+1})\) give
\begin{equation}
 \norm{Q_{j+1}u(\tau_j)}_2\le\widehat\delta_j.
 \label{eq:Qj1-upper}
\end{equation}
Lemma~\ref{lem:nested-hodge} shows that the nested projections commute and
that \(D_j=Q_j-Q_{j+1}\) is the orthogonal projection onto
\(\Hh(B_j)\cap\Hh(B_{j+1})^\perp\).  Hence
\begin{align}
 \norm{D_ju(\tau_j)}_2^2
 &=\norm{Q_ju(\tau_j)}_2^2-
 \norm{Q_{j+1}u(\tau_j)}_2^2\notag\\
 &\ge\theta_j-2\widehat\delta_j\sqrt{\theta_j}>0,
 \label{eq:Dj-lower}
\end{align}
Since \(D_j\) is an orthogonal projection,
\begin{equation}
 \ip{u(\tau_j)}{q_j}
 =\frac{\ip{u(\tau_j)}{D_ju(\tau_j)}}
        {\norm{D_ju(\tau_j)}_2}
 =\norm{D_ju(\tau_j)}_2>0.
 \label{eq:packet-pairing-positive}
\end{equation}
Thus \eqref{eq:Dj-lower} proves the pairing estimate in
\eqref{eq:hodge-full-mass}.  Difference spaces at distinct levels are
orthogonal.  Orthogonality to compactly supported solenoidal tests in
\(B_{j+1}\), followed by the local de Rham theorem and divergence freedom,
shows that \(q_j\) is a harmonic gradient there.

Finally, \eqref{eq:avoidance-selection}, the fact that \(q_j=0\) almost
everywhere outside \(B_j\), and \(\norm{q_j}_2=1\) imply
\eqref{eq:pre-avoidance}.  Take the explicit gap-two sequence \(n_j=2j\)
and retain the already defined tuples
\((\tau_{n_j},\sigma_{n_j},B_{n_j},\chi_{n_j},q_{n_j})\).  Lemma
\ref{lem:frozen-thinning} gives every asserted inheritance property,
including persistence under the later diagonal extraction.  No retained
packet is reprojected.
\end{proof}

\section{Two uniform linear estimates}
\label{sec:linear}

The common adjoint construction requires localization for a reverse Oseen
pulse and drift-independent smoothing for passive trajectories.  We record
both estimates with the dependencies used later.

\begin{lemma}[Drift-independent periodic Nash smoothing]
\label{lem:nash-smoothing}
Let \(b\) be smooth and divergence-free on \([s,t]\times\Omega\).  The
scalar propagator of
\[
 \partial_tf+b\cdot\nabla f=\nu\Delta f
\]
satisfies
\begin{align}
 \norm{S_b(t,s)}_{L^1\to L^2}
 &\le C_\Omega\bigl[1+(\nu(t-s))^{-3/4}\bigr],
 \label{eq:L1-L2}\\
 \norm{S_b(t,s)}_{L^2\to L^\infty}
 &\le C_\Omega\bigl[1+(\nu(t-s))^{-3/4}\bigr].
 \label{eq:L2-Linf}
\end{align}
The constant is independent of the size and derivatives of \(b\).
\end{lemma}

\begin{proof}
Kato's inequality and divergence freedom give the \(L^1\)-contraction
\(\norm{f(t)}_1\le\norm{f(s)}_1\).  Testing by \(f\) gives
\begin{equation}
 \frac12\frac{\dd}{\dd t}\norm{f(t)}_2^2
 +\nu\norm{\nabla f(t)}_2^2=0.
 \label{eq:scalar-energy}
\end{equation}
The periodic Nash inequality
\begin{equation}
 \norm{g}_2^{10/3}
 \le C_\Omega\norm{g}_1^{4/3}
 \bigl(\norm{\nabla g}_2^2+\norm{g}_2^2\bigr)
 \label{eq:periodic-nash}
\end{equation}
and the standard differential-inequality argument imply
\eqref{eq:L1-L2}; see Nash~\cite{Nash1958}.  Applying this estimate to the
adjoint scalar equation and using duality gives \eqref{eq:L2-Linf}.  The
transport term disappears from both \(L^1\) and \(L^2\) balances, so no drift
norm enters the constant.  Density extends the estimates from
\(L^1\cap L^2\) data to the stated spaces.
\end{proof}

\begin{lemma}[Uniform reverse-Oseen off-support escape]
\label{lem:oseen-escape}
Let \(w\) be the energy solution on \([0,L]\) of
\begin{equation}
 \partial_\rho w+(b\cdot\nabla)w+\nabla\pi=\nu\Delta w,
 \qquad \nabla\cdot w=0,
 \qquad w(0)=w_0\in L^2_\sigma(\Omega),
 \label{eq:reverse-oseen}
\end{equation}
where \(b\) is smooth and divergence-free and
\(\int_\Omega w_0\dd x=0\).  Put
\[
 W=\norm{w}_{L^\infty(0,L;L^2)},\qquad
 G=\norm{\nabla w}_{L^2((0,L)\times\Omega)},\qquad
 B=\norm{b}_{L^2(0,L;L^6)}.
\]
For every \(0\leq\zeta\leq1\) with
\(\zeta\in W^{2,\infty}(\Omega)\),
\begin{align}
 \int_\Omega\zeta\abs{w(L)}^2\dd x
 &\leq\int_\Omega\zeta\abs{w_0}^2\dd x
 +C_\Omega\norm{\nabla\zeta}_\infty
 L^{1/4}BW^{3/2}G^{1/2}\notag\\
 &\quad+\nu\norm{\Delta\zeta}_\infty LW^2.
 \label{eq:escape-general}
\end{align}
If \(B_{2r}(a)\) is contained in a Euclidean chart and
\(w_0\in\Hh(B_r(a))\), then
\begin{align}
 \norm{w(L)}_{L^2(\Omega\setminus B_{2r}(a))}^2
 &\leq C_\Omega r^{-1}(2\nu)^{-1/4}L^{1/4}B
       \norm{w_0}_2^2\notag\\
 &\quad+C_\Omega\nu r^{-2}L\norm{w_0}_2^2.
 \label{eq:escape-ball}
\end{align}
The constants do not depend on derivatives or pointwise norms of \(b\).
\end{lemma}

\begin{proof}
The spatial mean of \(w\) is preserved.  Testing
\eqref{eq:reverse-oseen} by \(w\) gives
\begin{equation}
 \norm{w(\rho)}_2^2
 +2\nu\int_0^\rho\norm{\nabla w(s)}_2^2\dd s
 =\norm{w_0}_2^2.
 \label{eq:reverse-energy}
\end{equation}
With the pressure normalized to have zero mean, taking divergence in
\eqref{eq:reverse-oseen} gives
\begin{equation}
 -\Delta\pi=\partial_i\partial_j(b_jw_i),
 \qquad
 \pi=\mathcal R_i\mathcal R_j(b_jw_i),
 \qquad
 \norm{\pi}_{3/2}\leq C_\Omega\norm{b}_6\norm{w}_2.
 \label{eq:pressure-estimates}
\end{equation}
For smooth data, multiplication by \(2\zeta w\), integration over the torus,
and divergence freedom give the exact localized identity
\begin{align}
 \frac{\dd}{\dd\rho}\int_\Omega\zeta\abs{w}^2\dd x
 +2\nu\int_\Omega\zeta\abs{\nabla w}^2\dd x
 ={}&\int_\Omega\abs{w}^2b\cdot\nabla\zeta\dd x\notag\\
 &+2\int_\Omega\pi w\cdot\nabla\zeta\dd x
 +\nu\int_\Omega\abs{w}^2\Delta\zeta\dd x.
 \label{eq:reverse-local-energy}
\end{align}
At each time,
\begin{align}
 \left|\int_\Omega\abs{w}^2b\cdot\nabla\zeta\dd x\right|
 &\leq\norm{\nabla\zeta}_\infty\norm{b}_6
       \norm{w}_{12/5}^2,
 \label{eq:reverse-transport-flux}\\
 \left|\int_\Omega\pi w\cdot\nabla\zeta\dd x\right|
 &\leq C_\Omega\norm{\nabla\zeta}_\infty\norm{b}_6
       \norm{w}_2\norm{w}_3.
 \label{eq:reverse-pressure-flux}
\end{align}
Because \(w\) has zero mean, the homogeneous periodic Sobolev and
interpolation inequalities yield
\begin{equation}
 \norm{w}_{12/5}^2
 \leq C_\Omega\norm{w}_2^{3/2}\norm{\nabla w}_2^{1/2},
 \qquad
 \norm{w}_3
 \leq C_\Omega\norm{w}_2^{1/2}\norm{\nabla w}_2^{1/2}.
 \label{eq:reverse-interpolation}
\end{equation}
Thus H\"older's inequality in time, with exponents \((2,4,4)\), gives
\begin{align}
 &\int_0^L\norm{b}_6
 \left(\norm{w}_{12/5}^2+\norm{w}_2\norm{w}_3\right)\dd\rho\notag\\
 &\qquad\leq C_\Omega L^{1/4}BW^{3/2}G^{1/2}.
 \label{eq:flux-bound}
\end{align}
Integrating \eqref{eq:reverse-local-energy}, dropping its nonnegative
left-hand dissipation, and using \eqref{eq:flux-bound} proves
\eqref{eq:escape-general}.

For the second assertion choose \(\zeta\) equal to zero on \(B_r(a)\), equal
to one on \(\Omega\setminus B_{2r}(a)\), and satisfying
\(\norm{\nabla\zeta}_\infty\leq C_\Omega r^{-1}\) and
\(\norm{\Delta\zeta}_\infty\leq C_\Omega r^{-2}\).  Since
\(w_0\in\Hh(B_r(a))\), the initial localized term vanishes.  From
\eqref{eq:reverse-energy},
\[
 W\leq\norm{w_0}_2,
 \qquad G\leq(2\nu)^{-1/2}\norm{w_0}_2.
\]
Substitution in \eqref{eq:escape-general} gives
\eqref{eq:escape-ball}.

For general \(L^2\) data, take smooth solenoidal Galerkin approximations.
The global energy identity gives strong convergence in
\(C([0,L];L^2)\) and weak convergence of the gradients; the localized
identity passes by the displayed bounds and lower semicontinuity.  In the
ball case, membership in \(\Hh(B_r(a))\) permits the initial approximants to
be chosen in \(C_{c,\sigma}^\infty(B_r(a))\), so the support property is
preserved throughout the approximation.  This completes the proof.
\end{proof}

\section{The atomic full-tail adjoint}
\label{sec:ghost}

We now construct one backward state from the whole packet tail.  For each
selected packet define
\begin{equation}
 v_j(t):=U(\tau_j,t)^*q_j,
 \qquad t\leq\tau_j.
 \label{eq:terminal-pulses}
\end{equation}
Thus \(v_j\) solves \eqref{eq:adjoint} and satisfies
\begin{equation}
 \norm{v_j(t)}_2^2
 +2\nu\int_t^{\tau_j}\norm{\nabla v_j(s)}_2^2\dd s=1.
 \label{eq:pulse-energy}
\end{equation}

\begin{proposition}[Compact preterminal extraction]
\label{prop:ghost-extraction}
After passing to a subsequence without changing the already selected
packets, there are a field \(A\), a number \(a_\infty>0\), and a number
\(d_A\in(0,1]\) such that the following statements hold.
On each compact interval \([t_b,S]\Subset[t_b,T_*)\),
\begin{align}
 v_j&\longrightarrow A
 &&\text{strongly in }L^2(t_b,S;L^2),
 \label{eq:ghost-strong-compact}\\
 v_j(t)&\weakto A(t)
 &&\text{in }L^2\text{ for every fixed }t<T_*,
 \label{eq:ghost-fixed-weak}
\end{align}
and the gradients converge weakly in the corresponding energy space.
The limit solves \eqref{eq:adjoint} and
\begin{align}
 \ip{u(t)}{A(t)}&=a_\infty,
 &&t_b<t<T_*,
 \label{eq:ghost-pairing}\\
 A(t)&\weakto0,
 &&t\uparrow T_*,
 \label{eq:ghost-weak-zero}\\
 \norm{A(t)}_2^2&\longrightarrow d_A,
 &&t\uparrow T_*.
 \label{eq:ghost-terminal-norm}
\end{align}
Moreover,
\begin{equation}
 a_\infty^2\geq m.
 \label{eq:ghost-lower-pairing}
\end{equation}
\end{proposition}

\begin{proof}
Fix \(S<T_*\).  For all sufficiently large \(j\), the energy identity
\eqref{eq:pulse-energy} bounds \(v_j\) in
\(L^\infty(t_b,S;L^2)\cap L^2(t_b,S;H^1)\).  Since \(u\) is smooth on this
compact interval, the equation bounds \(\partial_tv_j\) in
\(L^2(t_b,S;H^{-1})\).  Aubin--Lions compactness
\cite{Aubin1963,Simon1986}, a diagonal extraction, and weak continuity give
\eqref{eq:ghost-strong-compact}.  For a fixed smooth solenoidal test field
\(\varphi\), the scalar functions
\(t\mapsto\ip{v_j(t)}{\varphi}\) are equicontinuous on compact preterminal
intervals.  Almost-everywhere convergence following from
\eqref{eq:ghost-strong-compact}, followed by density, therefore gives
\eqref{eq:ghost-fixed-weak}.  Passing to the equation shows that \(A\) solves
\eqref{eq:adjoint}.

The Navier--Stokes state \(u\) itself solves the constrained equation
\eqref{eq:constrained} with initial datum \(u(t)\).  Duality consequently
gives
\begin{equation}
 \ip{u(t)}{v_j(t)}=\ip{u(\tau_j)}{q_j}.
 \label{eq:pulse-state-duality}
\end{equation}
The right-hand side is positive.  Passing to a further subsequence, it
converges to \(a_\infty\), and \eqref{eq:hodge-full-mass} implies
\eqref{eq:ghost-lower-pairing}.  Passing to the limit in
\eqref{eq:pulse-state-duality} proves \eqref{eq:ghost-pairing}.

For a fixed smooth solenoidal \(\varphi\), the adjoint equation and the
uniform \(L^2\) bound imply
\begin{equation}
 \left|\frac{\dd}{\dd t}\ip{v_j(t)}{\varphi}\right|
 \leq E_*\norm{\nabla\varphi}_\infty
      +\nu\norm{\Delta\varphi}_2=:C_\varphi.
 \label{eq:weak-trace-Lipschitz}
\end{equation}
The orthonormality of the packets gives \(q_j\weakto0\).  Integrating
\eqref{eq:weak-trace-Lipschitz} from \(t\) to \(\tau_j\) and then letting
\(j\to\infty\) yields
\begin{equation}
 \abs{\ip{A(t)}{\varphi}}
 \leq C_\varphi(T_*-t).
 \label{eq:weak-trace-rate}
\end{equation}
Density proves \eqref{eq:ghost-weak-zero}.

Finally, on each compact interval \([s,t]\subset[t_b,T_*)\), the limit
belongs to \(L^2(s,t;H^1)\cap L^\infty(s,t;L^2)\), and the adjoint equation
gives \(A_t\in L^2(s,t;H^{-1})\).  The Lions--Magenes time-continuity theorem
\cite{LionsMagenes1972}
therefore permits testing the equation by \(A\) and yields the exact backward
energy identity
\begin{equation}
 \norm{A(s)}_2^2
 +2\nu\int_s^t\norm{\nabla A(r)}_2^2\dd r
 =\norm{A(t)}_2^2,
 \qquad s<t<T_*.
 \label{eq:ghost-energy}
\end{equation}
The case \(s=t_b\) follows from the same theorem on \([t_b,t]\), or
equivalently by letting \(s\downarrow t_b\) and using the strong
\(L^2\)-continuity supplied by the energy class.  For every fixed \(t<T_*\),
\eqref{eq:ghost-fixed-weak}, weak lower semicontinuity, and
\eqref{eq:pulse-energy} give
\[
 \norm{A(t)}_2\leq\liminf_{j\to\infty}\norm{v_j(t)}_2\leq1.
\]
Moreover, \eqref{eq:ghost-energy} shows that
\(t\mapsto\norm{A(t)}_2^2\) is nondecreasing.  Hence the terminal norm limit
\(d_A\) exists and satisfies \(d_A\leq1\).  It is positive because
\eqref{eq:ghost-pairing}, \eqref{eq:ghost-lower-pairing}, and
\(\norm{u(t)}_2\leq E_*\) exclude \(A(t)\to0\) strongly.
\end{proof}

The next lemma is the only point at which spatial localization for the
reverse Oseen equation is needed.

\begin{lemma}[Terminal localization of the adjoint]
\label{lem:ghost-localization}
For every \(r>0\),
\begin{equation}
 \lim_{t\uparrow T_*}
 \norm{A(t)}_{L^2(\Omega\setminus B_r(a))}=0.
 \label{eq:ghost-off-atom}
\end{equation}
Consequently,
\begin{equation}
 \abs{A(t,x)}^2\dd x\weakstarto d_A\delta_a
 \qquad\text{as }t\uparrow T_*.
 \label{eq:ghost-measure-dA}
\end{equation}
\end{lemma}

\begin{proof}
Fix \(t<T_*\) and reverse the \(j\)-th pulse by setting
\begin{equation}
 w_j(\rho)=v_j(\tau_j-\rho),
 \qquad b_j(\rho)=-u(\tau_j-\rho),
 \qquad 0\leq\rho\leq\tau_j-t.
 \label{eq:reversed-pulse}
\end{equation}
The datum \(w_j(0)=q_j\) has zero mean and belongs to
\(\Hh(B_j)\), where \(B_j\downarrow\{a\}\).  The energy estimate and periodic Sobolev inequality
give, with \(\delta=T_*-t\),
\begin{equation}
 \norm{b_j}_{L^2_\rho L^6_x(0,\tau_j-t)}
 \leq C_\Omega\left(
   \norm{\nabla u}_{L^2_tL^2_x(T_*-\delta,T_*)}
   +E_*\delta^{1/2}\right).
 \label{eq:tail-drift-bound}
\end{equation}
It is enough to prove the assertion for radii small enough that
\(B_r(a)\) lies in one Euclidean chart; the general case follows by choosing
a smaller radius.  Fix such an \(r>0\).  Since \(B_j\downarrow\{a\}\), for
all sufficiently large \(j\) one has \(B_j\Subset B_{r/2}(a)\).  Nestedness
of the local solenoidal spaces gives
\(q_j\in\Hh(B_{r/2}(a))\).  Apply
Lemma~\ref{lem:oseen-escape} with initial radius \(r/2\).  Its right-hand side
is bounded by a quantity \(\omega_r(\delta)\) satisfying
\begin{equation}
 \omega_r(\delta)
 \leq C_{r,\nu,\Omega}\delta^{1/4}
 \left(
   \norm{\nabla u}_{L^2_tL^2_x(T_*-\delta,T_*)}
   +E_*\delta^{1/2}\right)
 +C_{r,\Omega}\nu\delta,
 \label{eq:off-atom-modulus}
\end{equation}
and \(\omega_r(\delta)\to0\) as \(\delta\downarrow0\).  Hence
\begin{equation}
 \norm{v_j(t)}_{L^2(\Omega\setminus B_r(a))}^2
 \leq\omega_r(T_*-t)
 \label{eq:pulse-off-atom-bound}
\end{equation}
for all sufficiently large \(j\).  Multiplication by
\(\one_{\Omega\setminus B_r(a)}\) is a bounded operator on \(L^2\), so
\eqref{eq:ghost-fixed-weak} and weak lower semicontinuity pass
\eqref{eq:pulse-off-atom-bound} directly to \(A(t)\).  Letting
\(t\uparrow T_*\) proves \eqref{eq:ghost-off-atom}.  Together with
\eqref{eq:ghost-terminal-norm}, this proves \eqref{eq:ghost-measure-dA}:
indeed, for \(\varphi\in C(\Omega)\),
\[
 \int_\Omega\varphi\abs{A(t)}^2\dd x
 -\varphi(a)\norm{A(t)}_2^2
 =\int_\Omega(\varphi-\varphi(a))\abs{A(t)}^2\dd x.
\]
Given \(\varepsilon>0\), choose \(r\) so that
\(\abs{\varphi(x)-\varphi(a)}<\varepsilon\) on \(B_r(a)\); the integral
outside \(B_r(a)\) tends to zero by \eqref{eq:ghost-off-atom}, while the
inside contribution is at most \(\varepsilon\norm{A(t)}_2^2\).  Letting
first \(t\uparrow T_*\) and then \(\varepsilon\downarrow0\) gives the claimed
weak-star limit.
\end{proof}

\begin{proposition}[Full-time Cauchy saturation]
\label{prop:cauchy-saturation}
The quantities in Proposition~\ref{prop:ghost-extraction} satisfy
\begin{equation}
 a_\infty=\sqrt m,
 \qquad d_A=1.
 \label{eq:cauchy-saturation}
\end{equation}
 In particular, \eqref{eq:full-mass-capture},
 \eqref{eq:constant-pairing}, \eqref{eq:adjoint-atom}, and
 \eqref{eq:atomic-alignment} hold.
\end{proposition}

\begin{proof}
Choose \(0\leq\chi_r\leq1\), with \(\chi_r=1\) on \(B_r(a)\) and
\(\operatorname{supp}\chi_r\subset B_{2r}(a)\).  Lemma
\ref{lem:ghost-localization} and \eqref{eq:ghost-pairing} give, for every
\(t<T_*\),
\[
 a_\infty=\ip{\chi_ru(t)}{A(t)}
 +\ip{(1-\chi_r)u(t)}{A(t)},
\]
and the second term tends to zero as \(t\uparrow T_*\).  Cauchy--Schwarz,
followed by the full-time convergence in
Lemma~\ref{lem:unique-endpoint-measure} and by
\eqref{eq:ghost-terminal-norm}, therefore gives
\begin{equation}
 a_\infty^2
 \leq d_A\int_\Omega\chi_r^2\dd\mu_*.
 \label{eq:local-cross-upper}
\end{equation}
Letting \(r\downarrow0\) gives \(a_\infty^2\leq md_A\).  On the other hand,
\eqref{eq:ghost-lower-pairing} gives \(m\leq a_\infty^2\), while
\(d_A\leq1\).  Thus
\begin{equation}
 m\leq a_\infty^2\leq md_A\leq m,
 \label{eq:saturation-chain}
\end{equation}
which proves \eqref{eq:cauchy-saturation}.  The asserted conclusions now
follow from Proposition~\ref{prop:ghost-extraction}, Lemma
\ref{lem:ghost-localization}, and the definition of \(a_\infty\).

It remains to record the strong atomic-layer statement with its order of
limits explicit.  Let \(0\leq\chi\leq1\) be smooth and equal to one near
\(a\).  Expanding the square gives
\begin{align}
 \int_\Omega\chi\abs{u-\sqrt m A}^2\dd x
 ={}&\int_\Omega\chi\abs{u}^2\dd x
 +m\int_\Omega\chi\abs{A}^2\dd x\notag\\
 &-2\sqrt m\int_\Omega\chi\,u\cdot A\dd x.
 \label{eq:alignment-square}
\end{align}
As \(t\uparrow T_*\), the three terms on the right converge respectively to
\(\int\chi\dd\mu_*\), \(m\), and \(2m\).  The last limit follows from the
constant pairing and
\(\norm{(1-\chi)A(t)}_2\to0\).  Hence
\begin{equation}
 \lim_{t\uparrow T_*}
 \int_\Omega\chi\abs{u(t)-\sqrt m A(t)}^2\dd x
 =\int_\Omega\chi\dd\mu_*-m.
 \label{eq:fixed-cutoff-alignment}
\end{equation}
Taking \(\chi=\chi_r\) yields
\begin{equation}
 \limsup_{t\uparrow T_*}
 \int_{B_r(a)}\abs{u(t)-\sqrt m A(t)}^2\dd x
 \leq\mu_*(B_{2r}(a))-m.
 \label{eq:alignment-ball-bound}
\end{equation}
The right-hand side tends to zero as \(r\downarrow0\), proving
\eqref{eq:atomic-alignment}.
\end{proof}

\section{From the adjoint to full-tail saturation}
\label{sec:conveyor}

We next upgrade compact-preterminal convergence to a uniform statement on
the whole late triangle.  This is a norm-saturation argument in one fixed
Hilbert space.

\begin{lemma}[Strong full-tail convergence]
\label{lem:strong-full-tail}
For every fixed \(t_0\in[t_b,T_*)\),
\begin{align}
 v_j(t_0)&\longrightarrow A(t_0)
 &&\text{strongly in }L^2,
 \label{eq:pulse-strong-fixed}\\
 \one_{(t_0,\tau_j)}\nabla v_j
 &\longrightarrow\nabla A
 &&\text{strongly in }L^2((t_0,T_*)\times\Omega).
 \label{eq:pulse-gradient-strong}
\end{align}
\end{lemma}

\begin{proof}
We first identify the weak limit on the complete terminal interval.  The
zero-extended gradients are uniformly bounded by
\eqref{eq:pulse-energy}.  On every \((t_0,S)\), \(S<T_*\), the extraction
in Proposition~\ref{prop:ghost-extraction} gives weak convergence of the
gradients to \(\nabla A\).  Let
\(\Phi\in L^2((t_0,T_*)\times\Omega)\).  After truncating \(\Phi\) at
\(S<T_*\), the compact-preterminal weak convergence applies, while
\begin{align}
 &\left|\ip{\one_{(t_0,\tau_j)}\nabla v_j-\nabla A}
                  {\Phi\one_{(S,T_*)}}\right|\notag\\
 &\quad\leq\left(
 \sup_j\norm{\one_{(t_0,\tau_j)}\nabla v_j}_{L^2_{t,x}}
 +\norm{\nabla A}_{L^2_{t,x}}\right)
 \norm{\Phi\one_{(S,T_*)}}_{L^2_{t,x}}.
 \label{eq:gradient-tail-test}
\end{align}
The last factor tends to zero as \(S\uparrow T_*\), uniformly in \(j\).
Hence
\begin{equation}
 \one_{(t_0,\tau_j)}\nabla v_j\weakto\nabla A
 \quad\text{in }L^2((t_0,T_*)\times\Omega).
 \label{eq:gradient-full-tail-weak}
\end{equation}

Next, the pulse identity gives the exact norm relation
\begin{equation}
 \norm{v_j(t_0)}_2^2
 +2\nu\norm{\one_{(t_0,\tau_j)}\nabla v_j}_{L^2_{t,x}}^2=1.
 \label{eq:pulse-total-one}
\end{equation}
Letting \(t\uparrow T_*\) in \eqref{eq:ghost-energy}, using monotone
convergence of the dissipation integral and \(d_A=1\), gives
\begin{equation}
 \norm{A(t_0)}_2^2
 +2\nu\norm{\nabla A}_{L^2((t_0,T_*)\times\Omega)}^2=1.
 \label{eq:ghost-total-one}
\end{equation}
Define
\begin{equation}
 Y_j=\left(v_j(t_0),\sqrt{2\nu}\,
          \one_{(t_0,\tau_j)}\nabla v_j\right),
 \qquad
 Y=\left(A(t_0),\sqrt{2\nu}\,\nabla A\right)
 \label{eq:direct-sum-Y}
\end{equation}
in
\(L^2_\sigma(\Omega)\oplus L^2((t_0,T_*)\times\Omega)\).  By
\eqref{eq:ghost-fixed-weak} and \eqref{eq:gradient-full-tail-weak},
\(Y_j\weakto Y\).  Equations \eqref{eq:pulse-total-one} and
\eqref{eq:ghost-total-one} give \(\norm{Y_j}=\norm{Y}=1\).  Therefore
\[
 \norm{Y_j-Y}^2
 =2-2\ip{Y_j}{Y}\longrightarrow0.
\]
The two orthogonal components converge strongly, proving
\eqref{eq:pulse-strong-fixed} and \eqref{eq:pulse-gradient-strong}.
\end{proof}

\begin{remark}[PDE input and Hilbert saturation]
\label{rem:pde-hilbert-split}
The terminal localization and Cauchy saturation in Section~\ref{sec:ghost}
are the PDE-specific steps that identify the weak limit and force its terminal
norm to be one.  Once those facts and the two exact energy identities are
available, Lemma~\ref{lem:strong-full-tail} is a Hilbert-space norm-saturation
argument on one fixed direct sum.  This separation is what converts compact
preterminal convergence into a uniform statement on the entire late
triangle.
\end{remark}

\begin{proof}[Proof of Theorem~\ref{thm:atom-conveyor}]
Apply Proposition~\ref{prop:catalog-selection} with any
\(\gamma_j\uparrow T_*\) and with
\(\delta_j^{\rm given}=\eta_j=2^{-j}\).  First perform the explicit
gap-two thinning and relabel the retained chain by \(j\).  Only after this
chain is frozen do we take the diagonal subsequence furnished by
Proposition~\ref{prop:ghost-extraction}.  Lemma~\ref{lem:frozen-thinning}
applies again to that further subsequence: every packet remains the original
packet, no projection is repeated, and the Hodge, harmonicity, mass, old-time
smallness, and passive-avoidance relations all persist after relabeling.
The catalogued geometry and full-mass lower estimate are contained in
Proposition~\ref{prop:catalog-selection}.
Proposition~\ref{prop:cauchy-saturation} upgrades the lower estimate to
\eqref{eq:full-mass-capture} and gives the asserted adjoint.

Fix once and for all the root time \(t_0=t_b\), independently of \(j\) and
\(k\).  For \(j<k\), the difference \(v_j-v_k\) solves the same homogeneous
adjoint equation on \([t_0,\tau_j]\).  Its energy identity is
\begin{align}
 &\norm{q_j-U(\tau_k,\tau_j)^*q_k}_2^2\notag\\
 &\quad=\norm{v_j(t_0)-v_k(t_0)}_2^2
 +2\nu\int_{t_0}^{\tau_j}
        \norm{\nabla(v_j-v_k)}_2^2\dd t.
 \label{eq:adjoint-difference-energy}
\end{align}
Introduce the Hilbert space
\[
 \mathcal X_{t_0}:=L^2_\sigma(\Omega)
 \oplus L^2((t_0,T_*)\times\Omega)
\]
and the direct-sum states
\[
 X_j:=\bigl(v_j(t_0),\sqrt{2\nu}\,
 \one_{(t_0,\tau_j)}\nabla v_j\bigr).
\]
Lemma~\ref{lem:strong-full-tail} says that \(X_j\) converges strongly in
\(\mathcal X_{t_0}\).  Since the right-hand side of
\eqref{eq:adjoint-difference-energy} is bounded above by
\(\norm{X_j-X_k}_{\mathcal X_{t_0}}^2\), it tends to zero uniformly for
\(k>j\geq J\) as \(J\to\infty\).  This proves
\eqref{eq:adjoint-conveyor}.  For each fixed \(j\), use
\eqref{eq:pulse-strong-fixed} with \(t_0=\tau_j\) and then let
\(k\to\infty\); this gives
\begin{equation}
 A(\tau_j)=\lim_{k\to\infty}U(\tau_k,\tau_j)^*q_k
 \quad\text{strongly in }L^2.
 \label{eq:moving-node-limit}
\end{equation}
Consequently,
\[
 \norm{A(\tau_j)-q_j}_2
 \leq\sup_{k>j}\norm{U(\tau_k,\tau_j)^*q_k-q_j}_2
 \leq\sup_{k>\ell\geq j}
 \norm{U(\tau_k,\tau_\ell)^*q_k-q_\ell}_2\longrightarrow0,
\]
which proves \eqref{eq:moving-locking} with the order of limits explicit.

Duality gives
\begin{equation}
 \operatorname{Re}\ip{U(\tau_k,\tau_j)q_j}{q_k}
 =\operatorname{Re}\ip{q_j}{U(\tau_k,\tau_j)^*q_k}
 \longrightarrow1
 \label{eq:forward-saturation}
\end{equation}
uniformly on the late triangle.  Contractivity and \(\norm{q_k}_2=1\) imply
\begin{equation}
 \norm{U(\tau_k,\tau_j)q_j-q_k}_2^2
 \leq2\left[1-
 \operatorname{Re}\ip{U(\tau_k,\tau_j)q_j}{q_k}\right],
 \label{eq:forward-error-control}
\end{equation}
which proves \eqref{eq:forward-conveyor} uniformly.

Let \(e_{j,k}^f=\norm{U(\tau_k,\tau_j)q_j-q_k}_2\) and
\(e_{j,k}^a=\norm{U(\tau_k,\tau_j)^*q_k-q_j}_2\).  The forward and
backward energy identities in Proposition~\ref{prop:oseen-family} give
\begin{align}
 2\nu\int_{\tau_j}^{\tau_k}
 \norm{\nabla U(t,\tau_j)q_j}_2^2\dd t
 &=1-\norm{U(\tau_k,\tau_j)q_j}_2^2
 \leq2e_{j,k}^f,
 \label{eq:forward-action-control}\\
 2\nu\int_{\tau_j}^{\tau_k}
 \norm{\nabla U(\tau_k,t)^*q_k}_2^2\dd t
 &=1-\norm{U(\tau_k,\tau_j)^*q_k}_2^2
 \leq2e_{j,k}^a.
 \label{eq:adjoint-action-control}
\end{align}
Taking the late-triangle suprema proves
\eqref{eq:vanishing-action} and \eqref{eq:vanishing-adjoint-action}.
\end{proof}

\begin{proof}[Proof of Corollary~\ref{cor:passive-separation}]
The explicit gap-two selection and Lemma~\ref{lem:frozen-thinning} give
\eqref{eq:passive-avoidance} without altering any packet used above.
Taking \(k=j+1\) in forward saturation gives
\(\operatorname{Re}\ip{U_jq_j}{q_{j+1}}\to1\), whereas
\eqref{eq:passive-avoidance} gives
\(\ip{S_jq_j}{q_{j+1}}\to0\).  Their difference is
\eqref{eq:leray-passive-separation}.
\end{proof}

\section{Second-order action forced by full-tail saturation}
\label{sec:action}

We now forget the endpoint atom and retain only one full-tail saturated family for the
fixed parent \(u\).  Fix a sufficiently late root \(J\) and set
\begin{equation}
 H(t)=H^J(t):=U(t,\tau_J)q_J,
 \qquad t\geq\tau_J.
 \label{eq:fixed-root-H}
\end{equation}
The local-Hodge construction gives zero spatial mean for \(q_J\), and the
constrained equation preserves it.  Thus the homogeneous periodic Sobolev
estimates used below are legitimate.

For \(j\geq J+1\), define
\begin{align}
 I_j&=[\tau_j,\tau_{j+1}],&
 \ell_j&=\abs{I_j},\notag\\
 d_j&=\int_{I_j}\norm{\nabla u(t)}_2^2\dd t,&
 a_j&=\int_{I_j}\norm{\nabla H(t)}_2^2\dd t,\notag\\
 K_j&=\int_{I_j}\norm{\Delta H(t)}_2^2\dd t,&
 \widetilde d_j&=d_j+E_*^2\ell_j.
 \label{eq:cell-actions}
\end{align}
The energy identities for \(u\) and \(H\) imply
\begin{equation}
 \sum_{j\geq J+1}d_j<\infty,\qquad
 \sum_{j\geq J+1}a_j<\infty,\qquad
 \sum_{j\geq J+1}\ell_j<\infty.
 \label{eq:first-order-summability}
\end{equation}

\begin{lemma}[Cellwise second-order action]
\label{lem:cell-H2-action}
There is \(J_0\) such that, for every fixed root \(J\geq J_0\) and every
\(j\geq J+1\),
\begin{equation}
 K_j\geq
 \min\left\{
 c_1\widetilde d_j^{-2}a_j^{-1},
 c_2\nu^{-2}\ell_j^{-1}
 \right\},
 \label{eq:H2-action-lower}
\end{equation}
where \(c_1>0\) depends only on the periodic Sobolev constant and
\(c_2>0\) is absolute.  The first entry is interpreted as \(+\infty\) when
\(a_j\widetilde d_j=0\).  In particular,
\begin{equation}
 K_j\longrightarrow\infty
 \qquad\text{as }j\to\infty.
 \label{eq:H2-action-divergence}
\end{equation}
\end{lemma}

\begin{proof}
Set
\begin{equation}
 \varepsilon_J^f=
 \sup_{k>j\geq J}\norm{U(\tau_k,\tau_j)q_j-q_k}_2.
 \label{eq:forward-tail-error}
\end{equation}
Choose \(J_0\) so that \(\varepsilon_J^f\leq1/8\) for \(J\geq J_0\), and
put
\begin{equation}
 c_0=\sqrt2-\frac14>0.
 \label{eq:fixed-displacement-constant}
\end{equation}
For \(j\geq J+1\), the cocycle law and forward saturation give
\[
 \norm{H(\tau_j)-q_j}_2\leq\varepsilon_J^f,
 \qquad
 \norm{H(\tau_{j+1})-q_{j+1}}_2\leq\varepsilon_J^f.
\]
Since consecutive packets are orthogonal,
\begin{equation}
 \norm{H(\tau_{j+1})-H(\tau_j)}_2
 \geq\norm{q_{j+1}-q_j}_2-2\varepsilon_J^f
 \geq c_0.
 \label{eq:fixed-root-displacement}
\end{equation}

The descendant solves
\begin{equation}
 H_t=-\Pp[(u\cdot\nabla)H]+\nu\Delta H.
 \label{eq:H-fixed-root-equation}
\end{equation}
The periodic Sobolev and interpolation inequalities give
\begin{equation}
 \norm{u}_6\leq C_\Omega(\norm{\nabla u}_2+\norm{u}_2),
 \qquad
 \norm{\nabla H}_3
 \leq C_\Omega\norm{\nabla H}_2^{1/2}
                    \norm{\Delta H}_2^{1/2}.
 \label{eq:cell-interpolation}
\end{equation}
Integrating \eqref{eq:H-fixed-root-equation} over \(I_j\), using
\(L^2\)-contractivity of \(\Pp\), and then applying H\"older in time with
exponents \((2,4,4)\), we obtain
\begin{align}
 \norm{H(\tau_{j+1})-H(\tau_j)}_2
 &\leq \int_{I_j}\norm{u}_6\norm{\nabla H}_3\dd t
       +\nu\int_{I_j}\norm{\Delta H}_2\dd t\notag\\
 &\leq C_\Omega\widetilde d_j^{1/2}a_j^{1/4}K_j^{1/4}
       +\nu\ell_j^{1/2}K_j^{1/2}.
 \label{eq:cell-displacement-upper}
\end{align}
Combining \eqref{eq:fixed-root-displacement} and
\eqref{eq:cell-displacement-upper}, at least one term on the final
right-hand side is at least \(c_0/2\).  Therefore
\[
 K_j\geq
 \min\left\{
 \left(\frac{c_0}{2C_\Omega}\right)^4
 \widetilde d_j^{-2}a_j^{-1},
 \frac{c_0^2}{4}\nu^{-2}\ell_j^{-1}
 \right\},
\]
which is \eqref{eq:H2-action-lower}.  By
\eqref{eq:first-order-summability},
\(\widetilde d_j\to0\), \(a_j\to0\), and \(\ell_j\to0\).  Both entries in
the minimum tend to \(+\infty\), proving
\eqref{eq:H2-action-divergence}.
\end{proof}

Recall \(\Gamma_J=\norm{\nabla H}_2^2\),
\(\Kk_J=\norm{\Delta H}_2^2\), and \(\mathcal P_J\) from
\eqref{eq:root-definitions}--\eqref{eq:production-definition}.  On every
strictly preterminal interval,
\begin{equation}
 \frac12\Gamma_J'(t)+\nu\Kk_J(t)=\mathcal P_J(t).
 \label{eq:fixed-root-enstrophy}
\end{equation}

\begin{corollary}[Unweighted positive production]
\label{cor:positive-production}
Every full-tail saturated Hodge family satisfies
\begin{equation}
 \int_{t_0}^{T_*}(\mathcal P_J(t))_+\dd t=\infty
 \label{eq:positive-production-infinite}
\end{equation}
for every sufficiently late fixed root \(J\), where \(t_0=\tau_{J+1}\).
\end{corollary}

\begin{proof}
For a finite prefix, \eqref{eq:fixed-root-enstrophy} gives
\begin{equation}
 \int_{t_0}^{\tau_{N+1}}\mathcal P_J(t)\dd t
 =\frac12\bigl[\Gamma_J(\tau_{N+1})-\Gamma_J(t_0)\bigr]
  +\nu\sum_{j=J+1}^{N}K_j.
 \label{eq:production-prefix}
\end{equation}
The last sum diverges by Lemma~\ref{lem:cell-H2-action}, while the terminal
enstrophy is nonnegative.  Hence the positive part cannot be integrable.
\end{proof}

\begin{proof}[Proof of Corollary~\ref{cor:parent-budget}]
For a sufficiently late fixed root \(J\), Lemma~\ref{lem:cell-H2-action}
gives
\begin{equation}
 \int_{\tau_{J+1}}^{T_*}
 \norm{\Delta U(t,\tau_J)q_J}_2^2\dd t
 =\sum_{j\geq J+1}K_j=\infty.
 \label{eq:root-infinite-budget}
\end{equation}
Since \(\norm{q_J}_2=1\), this proves
\eqref{eq:parent-budget-failure}; the atomless criterion follows by
contraposition.

Assume now \eqref{eq:serrin-budget-class}.  Fix
\(t_c\leq s<r<T_*\), let \(H(t)=U(t,s)h_s\), and normalize
\(\norm{h_s}_2=1\).  Proposition~\ref{prop:oseen-family} gives the uniform
positive-delay bound
\begin{equation}
 C_{\rm sm}:=
 \sup_{\norm{h_s}_2=1}\norm{\nabla U(r,s)h_s}_2^2<\infty.
 \label{eq:delayed-smoothing}
\end{equation}
The same positive delay justifies the enstrophy identity on every
\([r,R]\), \(R<T_*\), first for smooth data and then by approximation.

For \(3<q\leq\infty\), define
\begin{equation}
 \vartheta_q=\frac3q,\qquad
 r_q=\frac{2q}{q-2},\qquad
 p_q=\frac{2}{1-\vartheta_q}=\frac{2q}{q-3},\qquad
 \alpha_q=\frac{1+\vartheta_q}{1-\vartheta_q},
 \label{eq:serrin-exponents}
\end{equation}
with the usual values \(\vartheta_\infty=0\), \(r_\infty=2\), and
\(p_\infty=2\).  Since
\(1/q+1/r_q=1/2\), periodic interpolation gives
\begin{equation}
 \norm{\nabla H}_{r_q}
 \leq C_\Omega
 \norm{\nabla H}_2^{1-\vartheta_q}
 \norm{\Delta H}_2^{\vartheta_q}.
 \label{eq:serrin-gradient-interpolation}
\end{equation}
Writing
\(\Gamma=\norm{\nabla H}_2^2\) and
\(K=\norm{\Delta H}_2^2\), we obtain
\begin{align}
 \abs{\mathcal P_{s,h_s}(t)}
 &\leq\norm{\Delta H}_2\norm{u}_q\norm{\nabla H}_{r_q}\notag\\
 &\leq C_\Omega\norm{u}_q
       \Gamma^{(1-\vartheta_q)/2}
       K^{(1+\vartheta_q)/2}\notag\\
 &\leq\frac\nu2K
 +C_{\Omega,q}\nu^{-\alpha_q}\norm{u}_q^{p_q}\Gamma.
 \label{eq:budget-enstrophy-bound}
\end{align}
The exponents in the final Young inequality are
\(2/(1+\vartheta_q)\) and \(2/(1-\vartheta_q)=p_q\).

Set
\[
 f_q(t)=C_{\Omega,q}\nu^{-\alpha_q}\norm{u(t)}_q^{p_q}.
\]
The enstrophy identity and \eqref{eq:budget-enstrophy-bound} imply
\begin{equation}
 \Gamma'(t)+\nu K(t)\leq2f_q(t)\Gamma(t).
 \label{eq:budget-gronwall-inequality}
\end{equation}
The condition \(2/p+3/q\leq1\) is equivalent to \(p\geq p_q\).  Since the
time interval is finite, \(u\in L^p_tL^q_x\) implies
\(f_q\in L^1(r,T_*)\).  Gronwall's inequality and
\eqref{eq:delayed-smoothing} yield, uniformly over unit \(h_s\),
\begin{equation}
 \sup_{r\leq t<T_*}\Gamma(t)
 \leq C_{\rm sm}
 \exp\!\left(2\int_r^{T_*}f_q(t)\dd t\right)
 =:M_{u,\nu,q,s,r}<\infty.
 \label{eq:budget-uniform-Gamma}
\end{equation}
Integrating \eqref{eq:budget-gronwall-inequality} to \(R<T_*\) gives
\begin{equation}
 \nu\int_r^R K(t)\dd t
 \leq C_{\rm sm}
 +2M_{u,\nu,q,s,r}\int_r^{T_*}f_q(t)\dd t.
 \label{eq:budget-finite-serrin}
\end{equation}
Letting \(R\uparrow T_*\) and taking the supremum over unit \(h_s\) proves
\eqref{eq:parent-budget-criterion}.
\end{proof}

\section{Scope of the results}
\label{sec:scope}

The argument uses the smooth-preterminal Navier--Stokes energy identity, but
neither maximality of \(T_*\), a terminal value \(u(T_*)\), nor strong
\(L^2\) convergence.  The endpoint measure is obtained along the full time
variable, and one point atom is sufficient to trigger the construction.  The
result is same-parent at every stage: the catalogue, packets, forward
propagators, backward adjoint, and fixed-root descendants are all generated
by one unforced Navier--Stokes orbit.  No scale is assigned an independent
drift or an independently optimized adjoint.

Theorem~\ref{thm:atom-conveyor} is complementary to partial regularity and
critical-norm concentration.  Those theories constrain singular sets or the
size of norms near a possible singularity.  Here the output is an ordered
Hilbert-space genealogy, uniform on the full late triangle, together with a
second-order action cost.  Corollary~\ref{cor:parent-budget} removes the
atom-selected family from the resulting obstruction by passing to the
operator budget \(\mathfrak R_u\), which depends only on the parent
propagator.  The complete nonendpoint Serrin range in
\eqref{eq:serrin-budget-class} provides a natural class of sufficient
conditions for finiteness of that budget.

Periodicity is used in three identifiable places: the pressure representation
by periodic Riesz transforms, the nested local-Hodge construction inside
Euclidean chart balls, and the drift-independent periodic Nash estimate.
The saturation mechanism itself is formulated at the level of the Oseen
evolution family and its energy identities.  Analogues on domains with
boundary or for solutions available only through a local energy inequality
would therefore require corresponding pressure, local-Hodge, and evolution
estimates.  Appendix~\ref{sec:endpoint-tests} records the family-adapted
endpoint test-space consequence of the cellwise action lower bound.

\appendix

\section{Family-adapted endpoint tests}
\label{sec:endpoint-tests}

The cellwise divergence also forces discontinuity against a single endpoint
probe that remains fixed as finite prefixes increase.  This appendix
constructs that probe.  For the smooth preterminal function
\(\mathcal P_J\), its distributional derivative on \((t_0,T_*)\) is
normalized by
\begin{equation}
 \ip{\mathcal P_J'}{\varphi}
 :=-\int_{t_0}^{T_*}\mathcal P_J(t)\varphi'(t)\dd t,
 \qquad \varphi\in C_c^\infty(t_0,T_*).
 \label{eq:Pprime-definition}
\end{equation}

\begin{proposition}[Endpoint test-space obstruction]
\label{prop:endpoint-tests}
Assume \eqref{eq:energy-assumptions} and suppose that the same parent \(u\)
supports a full-tail saturated Hodge family.  For every sufficiently large
fixed root \(J\), set \(t_0=\tau_{J+1}\).  There exists a nonnegative weight
\begin{equation}
 B^J\in C([t_0,T_*])\cap W^{1,1}(t_0,T_*),
 \qquad (B^J)'\le0,
 \qquad B^J(T_*)=0,
 \label{eq:B-properties}
\end{equation}
depending on \(u\), the saturated family, \(J\), and the cutoff construction, but
not on a finite prefix \(N\), such that
\begin{equation}
 \int_{t_0}^{T_*}B^J(t)\Kk_J(t)\dd t=\infty,
 \qquad
 \int_{t_0}^{T_*}B^J(t)(\mathcal P_J(t))_+\dd t=\infty.
 \label{eq:weighted-divergence}
\end{equation}
Its primitive
\begin{equation}
 \kappa^J(t):=\int_t^{T_*}B^J(s)\dd s
 \label{eq:kappa-definition}
\end{equation}
belongs to \(C^1([t_0,T_*])\cap W^{2,1}(t_0,T_*)\) and satisfies
\begin{equation}
 \kappa^J(t)=o(T_*-t).
 \label{eq:kappa-small}
\end{equation}
There are terminally truncated primitives \(\kappa_N^J\), each vanishing on
a neighbourhood of \(T_*\), with
\begin{equation}
 \kappa_N^J\longrightarrow\kappa^J
 \quad\text{strongly in }C^1\cap W^{2,1},
 \label{eq:kappa-convergence}
\end{equation}
such that, for every fixed smooth left cutoff \(\eta\) that vanishes near
\(t_0\) and equals one on a terminal tail,
\begin{equation}
 \ip{\mathcal P_J'}{\eta\kappa_N^J}
 \ge \frac\nu2N-C_{J,\eta}\longrightarrow\infty.
 \label{eq:distribution-divergence}
\end{equation}
After reindexing the first retained cell, \(N\) in
\eqref{eq:distribution-divergence} is the number of retained cells.
Consequently \(\mathcal P_J'\) has no continuous extension to the endpoint
test spaces specified in Corollary~\ref{cor:endpoint-obstruction}.
\end{proposition}

\begin{corollary}[Endpoint obstruction forced by an atom]
\label{cor:endpoint-obstruction}
Under the hypotheses of Theorem~\ref{thm:atom-conveyor}, the conclusion of
Proposition~\ref{prop:endpoint-tests} holds for every sufficiently late fixed root.  In
particular, \(\mathcal P_J'\) cannot extend continuously to any of the
following spaces of endpoint-zero tests:
\begin{enumerate}[label=\textup{(\roman*)}]
\item the closure of \(C_c^\infty(t_0,T_*)\) in
\(W^{1,p}(t_0,T_*)\), for every \(1\leq p\leq\infty\);
\item the closure of \(C_c^\infty(t_0,T_*)\) in \(W^{2,1}(t_0,T_*)\);
\item the closure of \(C_c^\infty(t_0,T_*)\) in
\(C^1([t_0,T_*])\).
\end{enumerate}
Nor can \(\mathcal P_J'\) be represented by a finite signed measure on
\([t_0,T_*]\).  This obstruction is a further endpoint consequence of the atomic
full-tail rigidity.
\end{corollary}

\begin{lemma}[A summable weight sequence]
\label{lem:weight-sequence}
There are positive numbers \(\beta_j\) such that
\begin{equation}
 \sum_{j\geq J+1}\beta_j<\infty,
 \qquad
 \beta_jK_j\geq1
 \quad(j\geq J+1).
 \label{eq:weight-properties}
\end{equation}
\end{lemma}

\begin{proof}
Let \(c_1,c_2\) be the constants in \eqref{eq:H2-action-lower} and put
\begin{equation}
 \beta_j=c_1^{-1}\widetilde d_j^2a_j+c_2^{-1}\nu^2\ell_j.
 \label{eq:weight-definition}
\end{equation}
Set
\[
 A_j=c_1\widetilde d_j^{-2}a_j^{-1},
 \qquad C_j=c_2\nu^{-2}\ell_j^{-1},
\]
with \(A_j=+\infty\) when \(a_j\widetilde d_j=0\).  If
\(A_j\leq C_j\), then
\eqref{eq:H2-action-lower} gives
\[
 \beta_jK_j\geq c_1^{-1}\widetilde d_j^2a_j A_j=1.
\]
If \(C_j\leq A_j\), the same lower bound gives
\[
 \beta_jK_j\geq c_2^{-1}\nu^2\ell_j C_j=1.
\]
Thus \(\beta_jK_j\geq1\) in both alternatives, including \(a_j=0\).
Moreover,
\begin{equation}
 \sum_j\widetilde d_j^2a_j
 \leq\left(\sup_j\widetilde d_j^2\right)\sum_ja_j<\infty,
 \qquad
 \sum_j\ell_j<\infty,
 \label{eq:weight-summability}
\end{equation}
which proves the result.
\end{proof}

For each cell choose a terminal subinterval \(E_j\subset I_j\) so short that
\begin{equation}
 \int_{E_j}\Kk_J(t)\dd t\leq\frac12K_j.
 \label{eq:low-K-tail}
\end{equation}
Let \(\vartheta_j\) be a smooth nonincreasing function which equals one
before \(E_j\), decreases from one to zero inside \(E_j\), vanishes after
\(\tau_{j+1}\), and is flat at the transition endpoints.  Define
\begin{equation}
 B^J(t):=\sum_{j\geq J+1}\beta_j\vartheta_j(t).
 \label{eq:B-sum}
\end{equation}
Uniform convergence and disjointness of the transition intervals give
\begin{equation}
 B^J\in C([t_0,T_*])\cap W^{1,1}(t_0,T_*),
 \quad (B^J)'\leq0,
 \quad B^J(T_*)=0,
 \quad
 \norm{(B^J)'}_1=\sum_{j\geq J+1}\beta_j.
 \label{eq:B-regularity-proof}
\end{equation}
On \(I_j\setminus E_j\), all future summands remain equal to one, so
\begin{equation}
 B^J(t)\geq\sum_{k\geq j}\beta_k\geq\beta_j.
 \label{eq:B-cell-lower}
\end{equation}
It follows from \eqref{eq:low-K-tail} and \eqref{eq:weight-properties} that
\begin{equation}
 \int_{I_j}B^J(t)\Kk_J(t)\dd t\geq\frac12.
 \label{eq:weighted-K-per-cell}
\end{equation}
This proves the first divergence in \eqref{eq:weighted-divergence}.

For a finite prefix put
\begin{equation}
 B_N^J(t):=\sum_{j=J+1}^{N}\beta_j\vartheta_j(t),
 \qquad
 \kappa_N^J(t):=\int_t^{T_*}B_N^J(s)\dd s.
 \label{eq:finite-endpoint-tests}
\end{equation}
Also set
\begin{equation}
 \kappa^J(t):=\int_t^{T_*}B^J(s)\dd s.
 \label{eq:infinite-endpoint-test}
\end{equation}
The weight summability gives
\begin{align}
 \norm{B_N^J-B^J}_\infty
 &\leq\sum_{j>N}\beta_j\longrightarrow0,
 \label{eq:B-uniform-convergence}\\
 \norm{(B_N^J)'-(B^J)'}_1
 &=\sum_{j>N}\beta_j\longrightarrow0.
 \label{eq:B-derivative-convergence}
\end{align}
Consequently,
\begin{equation}
 \kappa_N^J\longrightarrow\kappa^J
 \quad\text{strongly in }
 C^1([t_0,T_*])\cap W^{2,1}(t_0,T_*).
 \label{eq:kappa-strong-proof}
\end{equation}
Since \(B^J\) is nonincreasing and tends to zero,
\begin{equation}
 0\leq\frac{\kappa^J(t)}{T_*-t}
 \leq B^J(t)\longrightarrow0.
 \label{eq:kappa-little-o-proof}
\end{equation}

\begin{lemma}[Finite-prefix lower bound]
\label{lem:finite-prefix-weight}
There is \(C_J<\infty\) such that
\begin{equation}
 \int_{t_0}^{T_*}B_N^J(t)\mathcal P_J(t)\dd t
 \geq\frac{\nu}{2}(N-J)-C_J.
 \label{eq:weighted-production-prefix}
\end{equation}
In particular,
\begin{equation}
 \int_{t_0}^{T_*}B^J(t)(\mathcal P_J(t))_+\dd t=\infty.
 \label{eq:weighted-positive-production}
\end{equation}
\end{lemma}

\begin{proof}
The functions \(B_N^J\) vanish on a neighbourhood of \(T_*\), although they
need not vanish at \(t_0\).  Integration by parts in
\eqref{eq:fixed-root-enstrophy} gives
\begin{align}
 \int_{t_0}^{T_*}B_N^J\mathcal P_J\dd t
 &=\frac12[B_N^J\Gamma_J]_{t_0}^{T_*}
   -\frac12\int_{t_0}^{T_*}(B_N^J)'\Gamma_J\dd t\notag\\
 &\qquad+\nu\int_{t_0}^{T_*}B_N^J\Kk_J\dd t.
 \label{eq:weighted-enstrophy-identity}
\end{align}
The middle term is nonnegative.  For every retained cell
\(J+1\leq j\leq N\), the argument in
\eqref{eq:weighted-K-per-cell} applies with \(B_N^J\geq\beta_j\) outside
\(E_j\).  Thus the last term is at least
\(\frac{\nu}{2}(N-J)\), while the left boundary is uniformly bounded.  This
proves \eqref{eq:weighted-production-prefix}.  Since \(0\leq B_N^J\leq B^J\),
the left-hand side is bounded above by
\(\int B^J(\mathcal P_J)_+\).  Letting \(N\to\infty\) proves
\eqref{eq:weighted-positive-production}.
\end{proof}

\begin{proof}[Proof of Proposition~\ref{prop:endpoint-tests}]
Equations \eqref{eq:B-regularity-proof},
\eqref{eq:weighted-K-per-cell}, \eqref{eq:weighted-positive-production},
\eqref{eq:kappa-strong-proof}, and \eqref{eq:kappa-little-o-proof} give all
assertions through \eqref{eq:kappa-convergence}.

Fix \(t_0<t_1<T_*\) and choose a smooth left cutoff \(\eta\) that vanishes
near \(t_0\) and equals one on \([t_1,T_*)\).  The product
\(\eta\kappa_N^J\) is a legitimate compactly supported test on
\((t_0,T_*)\), and the definition of the distributional derivative gives
\begin{equation}
 \ip{\mathcal P_J'}{\eta\kappa_N^J}
 =\int_{t_0}^{T_*}\eta B_N^J\mathcal P_J\dd t
  -\int_{t_0}^{T_*}\eta'\kappa_N^J\mathcal P_J\dd t.
 \label{eq:Pprime-cutoff-identity}
\end{equation}
The difference between the first integral here and
\(\int B_N^J\mathcal P_J\) is supported in the fixed compact interval
\([t_0,t_1]\).  Both that difference and the second integral are uniformly
bounded in \(N\), because \(B_N^J\) and \(\kappa_N^J\) converge uniformly
there and \(\mathcal P_J\) is smooth preterminally.  Lemma
\ref{lem:finite-prefix-weight}, followed by relabeling \(N\) as the number of
retained cells, therefore gives
\begin{equation}
 \ip{\mathcal P_J'}{\eta\kappa_N^J}
 \geq\frac{\nu}{2}N-C_{J,\eta}\longrightarrow\infty.
 \label{eq:Pprime-divergence-proof}
\end{equation}

The functions \(\eta\kappa_N^J\) are smooth and compactly preterminal and
converge strongly to \(\eta\kappa^J\) in each of the three test spaces listed
in Corollary~\ref{cor:endpoint-obstruction}.  A continuous extension of
\(\mathcal P_J'\) to any one of those spaces would therefore make the
pairings converge, contradicting \eqref{eq:Pprime-divergence-proof}.  A
finite signed measure on the closed interval would likewise act
continuously under uniform convergence.  This proves the proposition and its
corollary.
\end{proof}

\begin{remark}[Dependence of the endpoint weight]
\label{rem:family-adapted-weight}
The construction is family-adapted: \(B^J\) depends on the given parent,
saturated family, root, and selected low-\(\Kk_J\) terminal pieces \(E_j\),
while remaining fixed with respect to the prefix \(N\).  It records the
endpoint irregularity associated with the selected saturated family.
\end{remark}

\bibliographystyle{amsplain}
\bibliography{references}

\end{document}